\documentclass[11pt,reqno]{amsart}
\usepackage{enumerate}
\usepackage{amsrefs}
\usepackage{amsfonts, amsmath, amssymb, amscd, amsthm, bm, cancel}
\usepackage{url}
\usepackage{graphicx}
\usepackage[
linktocpage=true,colorlinks,citecolor=magenta,linkcolor=blue,urlcolor=magenta]{hyperref}
\usepackage{multicol}
\usepackage{comment}
\usepackage[margin=1in]{geometry}
\newtheorem{thm}{Theorem}[section]
\newtheorem*{thmA}{Theorem A}
\newtheorem*{thmB}{Theorem B}
\newtheorem*{thmC}{Theorem C}

 \newtheorem{cor}{Corollary}[section]

  \newtheorem{lem}{Lemma}[section]

 \newtheorem{prop}{Proposition}[section]
 \theoremstyle{definition}
 
  \newtheorem*{ack}{Acknowledgments}
 \theoremstyle{remark}
\newtheorem{rem}{Remark}[section]

 \numberwithin{equation}{section}

\allowdisplaybreaks

\newcommand{\f}{\left(}
\renewcommand{\r}{\right)}

\newcommand{\R}{\mathbb{R}}

\renewcommand{\a}{\alpha}
\renewcommand{\b}{\beta}
\newcommand{\g}{\varphi}
\renewcommand{\d}{\delta}

\renewcommand{\k}{\kappa}
\renewcommand{\l}{\lambda}

\newcommand{\s}{\sigma}

\newcommand{\metric}[2]{\ensuremath{\langle #1, #2\rangle}}  

\begin{document}
\title{Asymptotic convergence for a class of anisotropic curvature flows}

\author{Haizhong Li}
\address{Department of Mathematical Sciences, Tsinghua University, Beijing 100084, P.R. China}
\email{\href{mailto:lihz@tsinghua.edu.cn}{lihz@tsinghua.edu.cn}}

\author{BOTONG XU}
\address{Department of Mathematical Sciences, Tsinghua University, Beijing 100084, P.R. China}
\email{\href{mailto:xbt17@mail.tsinghua.edu.cn}{xbt17@mail.tsinghua.edu.cn}}

\author{RUIJIA ZHANG}
\address{Department of Mathematical Sciences, Tsinghua University, Beijing 100084, P.R. China}
\email{\href{mailto:zhangrj17@mails.tsinghua.edu.cn}{zhangrj17@mails.tsinghua.edu.cn}}



\begin{abstract}
In this paper, by using new auxiliary functions, we study a class of contracting flows of closed, star-shaped hypersurfaces in $\mathbb{R}^{n+1}$ with speed $r^{\frac{\a}{\b}}\s_k^{\frac{1}{\b}}$, where $\s_k$ is the $k$-th elementary symmetric polynomial of the principal curvatures, $\a$, $\b$ are positive constants and $r$ is the distance from points on the hypersurface to the origin. We obtain convergence results under some assumptions of $k$, $\a$, $\b$. When $k\geq2$, $0<\b\leq 1$, $\a\geq \b+k$,  we prove that the $k$-convex solution to the flow exists for all time and converges smoothly to a sphere after normalization, in particular, we generalize Li-Sheng-Wang's result from uniformly convex to $k$-convex. When $k \geq 2$, $\b=k$, $\a \geq 2k$, we prove that the $k$-convex solution to the flow exists for all time and converges smoothly to a sphere after normalization, in particular, we generalize Ling Xiao's result from $k=2$ to $k \geq 2$.
\end{abstract}

\maketitle

\section{introduction}\label{sec:1}
The classical Brunn-Minkowski theory plays an important role in the study of convex geometry and develops rapidly in recent years. 
Area measures and curvature measures are basic concepts in this theory. The classical Minkowski problem is a problem of prescribing a given $n$th area measure, which was finally settled by Cheng and Yau \cite{CY76}. While, the Christoffel-Minkowski problem is a problem of prescribing a given $k$-th area measure. In \cite{GM03}, Guan and Ma gave a sufficient condition for the existence of convex solution to this problem.

In 1962, Firey \cite{F62} generalized the Minkowski combination to p-sums from $p=1$ to $p\geq1$, which was used for developing a $L_p$ Brunn-Minkowki theory (see \cite{L93}). Later on, Huang, Lutwak, Yang and Zhang \cite{HLYZ16} (also see \cite{LYZ18}) introduced the dual $L_p$-Minkowski problem.
This sort of problem aims to find a convex body $\Omega$ satisfying the following equation
\begin{align}\label{lp mink}
r^{-(n+1-q)} u^{1-p}\f x \r \det\f D^2u+uI\r=f(x) \quad {\rm on} \ \mathbb{S}^n,
\end{align}
where $f$ is a given positive function on $\mathbb{S}^n$, $u$ is the support function of $\Omega$, $D$ is the Levi-Civita connection on $\mathbb{S}^n$ and $r=\sqrt{u^2 +|D u|^2}$. When $q=n+1$, it is the $L_p$-Minkowski problem (see, e.g., \cite{CW06,LO95, LW13, Zhu15}). When $p=0$, it corresponds to the dual Minkowski problem (see, e.g.,  \cite{HLYZ16, Zhao18}).
Meanwhile, Guan, Lin and Ma \cite{GLM09} considered the following prescribed $(n-k)$-th curvature measure problem
\begin{align}\label{dual mink}
r^{n+1} u^{-1}\s_k(\kappa) =f,
\end{align}
where $f\in C^2(\mathbb{S}^n)$ 
 and $\s_k(\k)$ is defined as the $k$-th elementary symmetric function of principal curvatures $\kappa=(\kappa_1,\cdots,\kappa_n)$ of $\partial\Omega$, i.e.
\begin{equation*}
\s_m(\k)=\sum_{1\leq i_1<\cdots<i_m \leq n}\k_{i_1}\cdots \k_{i_m}.
\end{equation*}
Combining \eqref{lp mink} and \eqref{dual mink}, it's interesting to consider the following prescribed curvature problem
\begin{align}\label{elliptic}
r^{\a} u^{-\b}\s_k(\kappa) =f,
\end{align}
where $\alpha$, $\beta$ are constants.
When $f$ is a constant and $\b=1$, the convex solution to equation \eqref{elliptic} has been well studied (see, e.g., \cite{BCD17, GaoLM18, LSW20a, LSW20b})  . Therefore, it is natural for us to consider the following equation in this paper, i.e. $f\equiv c$ in \eqref{elliptic}
\begin{align}\label{elliptic2}
r^{\a} u^{-\b}\s_k(\kappa) =c.
\end{align}

Curvature flows are effective tools to study the above fully nonlinear partial differential equations (see, e.g., \cite{And00, BIS20, CW00, Iva19, LL20, LSW20a, LSW20b, Xiao20}). Let $X_0$ be a closed, star-shaped hypersurface, we will consider the following curvature flow.
Let $X:M^n \times [0,T) \rightarrow\mathbb{R}^{n+1}$ be a family of smooth closed hypersurface, which satisfies
\begin{equation}\label{flow-s}
\left\{\begin{aligned}
\frac{\partial}{\partial t}X(x,t)=& -r^{\frac{\a}{\b}}\s_k(\kappa)^{\frac{1}{\b}} \nu(x,t), \\
X(\cdot,0)=& X_0(\cdot),
\end{aligned}\right.
\end{equation}
where $\b>0$, $1 \leq k \leq n$, $\nu$ is the unit outer normal vector at $X(\cdot,t)$ and $\s_k(\k)$ is the $k$-th mean curvature defined on $M_t=X(\cdot,t)$. For $\b=1$, Li, Sheng and Wang \cite{LSW20b} proved the following result.

\begin{thmA}[\cite{LSW20b}]\label{A}
Let $M_0$ be a smooth, closed, uniformly convex hypersurface in $\mathbb{R}^{n+1}$ enclosing the origin. If $\b=1$, $\a \geq k+1$, the flow \eqref{flow-s} has a unique smooth and uniformly convex solution $M_t$ for all time $t>0$, which converges to the origin. After a proper rescaling $X \rightarrow \phi^{-1}(t) X$, where
\begin{equation}\label{lsw}
	\begin{aligned}
	\phi(t) =& e^{-\gamma t}, \ &{\rm if} \ \a =k+1,\\
	\phi(t) =& \f 1+(\a-k-1)\gamma t \r^{-\frac{1}{\a-k-1}}, \ &{\rm if} \ \a \neq k+1,
	\end{aligned}
\end{equation}	
and $\gamma=C_n^k$,	the hypersurface $\tilde{M_t} = \phi^{-1}(t) M_t$ converges exponentially to a sphere centered at the origin in the $C^\infty$ topology.
\end{thmA}

Notice that the solutions to \eqref{flow-s} are not necessary uniformly convex (i.e. $\kappa_i> 0$, $i=1,\cdots,n$). 
We call a solution  $k$-convex if $\kappa\in\Gamma_k^+$ everywhere on $M_t$ for all $t \in [0,T)$, where
\begin{align*}
\Gamma_k^+=\lbrace\f \kappa_1,\cdots, \kappa_n\r\in\mathbb{R}^n |\s_1(\kappa)>0,\cdots,\s_k(\kappa)>0\rbrace.
\end{align*}
When $k=1$, we also call it mean-convex.
In our first theorem, using a new auxiliary function \eqref{aux-fun-2} in Lemma \ref{Lemma 4.5}, we generalize the assumption of the initial hypersurface in Theorem A from uniformly convex to $k$-convex.
\begin{thm}\label{main-thm-3}
Let $X_0: M^n \to \mathbb \R^{n+1} $ be a smooth, closed, $k$-convex and star-shaped hypersurface in $\mathbb R^{n+1}$. When $k\geq2$, $0<\beta\leq 1$ and $\a\geq k+\b$, the flow \eqref{flow-s} has a unique smooth and $k$-convex solution for all time $t\in [0,\infty)$ and converges smoothly to the origin. After a proper rescaling $X \to \phi^{-1}(t) X$, where $\phi$ is given by \eqref{def-phi}, the hypersurface $\tilde{M_t} = \phi^{-1}(t) M_t$ converges exponentially to a sphere centered at the origin in the $C^\infty$ topology.
\end{thm}
\begin{rem}
	When $\b=1$, $X_0$ is uniformly convex, Theorem \ref{main-thm-3} generalizes Li-Sheng-Wang's Theorem A.
\end{rem}

In \cite{Xiao20}, Ling Xiao studied the $2$-convex solution to the flow \eqref{flow-s} with $k=2$, $\b=2$ and $\alpha\geq 4$. In fact, she proved
\begin{thmB}[\cite{Xiao20}]\label{B}
Let $M_0$ be a smooth, closed, $2$-convex, star-shaped hypersurface in $\mathbb{R}^{n+1}$. If $k=2$, $\b=2$, $\a \geq 4$, the flow \eqref{flow-s} has a unique smooth star-shaped solution $M_t$ with positive scalar curvature, for all time $t>0$, which converges to the origin.
After a proper rescaling $X \rightarrow \phi^{-1}(t) X$, where
	\begin{align*}
	\phi(t) =& e^{-\gamma t}, \quad &\a =4,\\
	\phi(t) =& \f 1+\frac{(\a-4)\gamma}{2} t \r^{-\frac{2}{\a-4}}, \quad &\a > 4,
	\end{align*}
	and $\gamma = \sqrt{\frac{n(n-1)}{2}}$, the hypersurface $\tilde{M_t} = \phi^{-1}(t) M_t$ converges exponentially fast to a sphere centered at the origin in the $C^\infty$ topology.
\end{thmB}

The main difficulty to prove Theorem B is the $C^2$ estimates of flow \eqref{flow-s}. When $k=2$, Ling Xiao \cite{Xiao20} introduced a  technique to give the bounds of the principal curvatures under the normalized flow of \eqref{flow-s}. However, Ling Xiao's technique does not work for $k \geq 3$. 
In this paper, by using the new auxiliary function \eqref{aux-fun-2} in Lemma \ref{Lemma 4.5}, we generalize Ling Xiao's Theorem B from $k=2$ to $2 \leq  k \leq n$.
\begin{thm}\label{main-thm-2}
Let $X_0: M^n \to \mathbb \R^{n+1} $ be a smooth, closed, $k$-convex and star-shaped hypersurface in $\mathbb R^{n+1}$. Assume $\beta=k$, $k\geq2$ and $\a\geq 2k$, the flow \eqref{flow-s} has a unique smooth and $k$-convex solution for all time $t\in [0,\infty)$ and converges smoothly to the origin. After a proper rescaling $X \to \phi^{-1}(t) X$, where $\phi$ is given by \eqref{def-phi}, the hypersurface $\tilde{M_t} = \phi^{-1}(t) M_t$ converges exponentially to a sphere centered at the origin in the $C^\infty$ topology.
\end{thm}
\begin{rem}
	When $k=2$, Theorem \ref{main-thm-2} reduces to Ling Xiao's Theorem B.
\end{rem}
In \cite{LSW20b}, the authors also studied the flow with mean-convex initial hypersurface and obtained
\begin{thmC}[\cite{LSW20b}]\label{C}
Let $M_0$ be a smooth, closed and weakly mean-convex hypersurface in $\mathbb{R}^{n+1}$. Suppose that $M_0$ is star-shaped with respect to the origin. Assume $k=1$, $\b=1$ and $\a \geq 2$, then the flow \eqref{flow-s} has a unique smooth solution $M_t$ for all time $t>0$, and $M_t$ converges to the origin. After a proper rescaling $X \to \phi^{-1}(t) X$, where $\phi$ is given by \eqref{lsw}, the hypersurface $\tilde{M_t} = \phi^{-1}(t)M_t$ converges exponentially to a sphere centered at the origin in the $C^\infty$ topology.
\end{thmC}	
In Theorem C, $k=1$ and $\b=1$, the flow \eqref{flow-s} relates to a quasilinear parabolic equation and the long time existence follows from $C^0$, $C^1$ estimates. However, if $k=1$ and $\b \neq 1$, the corresponding equation is fully nonlinear. In our paper, using a new auxiliary function \eqref{aux-fun-1} in Lemma \ref{Lem 4.3}, we give a proof of $C^2$ estimates of the normalized flow of \eqref{flow-s} and then obtain the following theorem.
\begin{thm}\label{main-thm-1}
Assume $n\geq 2$, let $X_0: M^n \to \mathbb \R^{n+1} $ be a smooth, closed, mean-convex and star-shaped hypersurface in $\mathbb R^{n+1}$. Assume $k=1$, $\b>0$, $\a\geq \b+1$ the flow \eqref{flow-s} has a unique smooth and mean-convex solution for all time $t\in [0,\infty)$ and it converges smoothly to the origin. After a proper rescaling $X \to \phi^{-1}(t) X$, where $\phi$ is given by \eqref{def-phi}, the hypersurface $\tilde{M_t} = \phi^{-1}(t) M_t$ converges exponentially to a sphere centered at the origin in the $C^\infty$ topology.
\end{thm}
\begin{rem}
	When $k=1$, $\b=1$, Theorem \ref{main-thm-1} reduces to Theorem C.
\end{rem}
In our paper, we also generalize Theorem A from $\b=1$ to $\b>0$ with uniformly convex initial hypersurface.
\begin{thm}\label{main-thm-4}
	Let $M_0$ be a smooth, closed, uniformly convex hypersurface in $\mathbb{R}^{n+1}$($n\geq2$) enclosing the origin. If $\beta>0$, $\a \geq \beta+k$, the flow \eqref{flow-s} has a unique smooth and uniformly convex solution $M_t$ for all time $t\in [0,\infty)$ and it converges smoothly to the origin. After a proper rescaling $X \to \phi^{-1}(t) X$, where $\phi$ is given by \eqref{def-phi}, the hypersurface $\tilde{M_t} = \phi^{-1}(t) M_t$ converges exponentially to a sphere centered at the origin in the $C^\infty$ topology.
\end{thm}
\begin{rem}
	when $\b=1$, Theorem \ref{main-thm-4} reduces to Theorem A.
\end{rem}
Analogous to the calculations in \cite{LSW20b, Xiao20}, we can derive the normalized flow of \eqref{flow-s}. For convenience, we denote $\Phi=r^{\frac{\a}{\b}}\s_k^{\frac{1}{\b}}$ and $F=\s_k^{\frac{1}{\b}}$. Define $\tilde{X}(t) = \phi^{-1}(t) X(t)$, where
\begin{equation}\label{def-phi}	
\phi(t) = \left\{\begin{split}
  &e^{-\gamma t}, \quad &{\rm if} \ \a =k+\b,\\
 &\f 1+\frac{(\a-\b-k) \gamma}{\b} t \r^{\frac{\b}{k+\b-\a}}, \quad &{\rm if} \ a \neq k+\b,
\end{split}
\right.
\end{equation}
and $\gamma= ({C_n^k})^{\frac{1}{\b}}$.
Letting $\tilde{r} =\phi^{-1}(t) r$, $\tilde{F} = \phi^{\frac{k}{\b}}(t) F$, we have
\begin{align*}
\frac{\partial }{\partial t} \tilde{X} =& \partial_t (\phi^{-1}(t)) X + \phi^{-1}(t) \partial_t X \\
	=&\gamma\phi^{\frac{\a-\b-k}{\b}}(t) \tilde{X} - \phi^{-1}(t)(\phi(t) \tilde{r})^{\frac{\a}{\b}} \phi^{-\frac{k}{\b}}(t)\tilde{F}\nu \\
	=& \gamma\phi^{\frac{\a-\b-k}{\b}}(t) \tilde{X} - \phi^{\frac{\a-\b-k}{\b}}(t) \tilde{r}^{\frac{\a}{\b}} \tilde{F} \nu.
\end{align*}
Now we define $\tau$ to be
\begin{equation*}
\tau(t) = \left\{\begin{split}
&t, \quad &{\rm if} \ \a =k+\b,\\
&\frac{\b \log \f \b+\f \a-\b-k\r \gamma t \r-\b\log\b}{ \f \a-\b-k \r \gamma}, \quad &{\rm if} \ a \neq k+\b.
\end{split}
\right.
\end{equation*}
Then we have
\begin{align*}
\frac{\partial \tilde{X}}{\partial \tau} = \gamma\tilde{X} - \tilde{r}^{\frac{\a}{\b}} \tilde{F}\nu.
\end{align*}

From now on, we  call
\begin{equation}\label{s1:flow-n}
\left\{\begin{aligned}
\frac{\partial}{\partial t}X(x,t)=& -r^{\frac{\a}{\b}}\s_k(\kappa)^{\frac{1}{\b}} \nu(x,t)+\gamma   X,
\\
X(x,0) =&X_0(x),
\end{aligned}\right.
\end{equation}
the normalized flow of \eqref{flow-s}. 


Notice that the  solution to  $r^{\a} u^{-\b}\s_k(\kappa) =\gamma$ remains invariant under flow \eqref{s1:flow-n}. Theorem \ref{main-thm-3}, Theorem \ref{main-thm-2}, Theorem \ref{main-thm-1} and Theorem \ref{main-thm-4} imply the following result of the prescribed curvature equation \eqref{elliptic2}.
\begin{prop}
When $k\geq 2$, $\a\geq\b+k$, either $\b=k$ or $0<\b\leq 1$, the smooth, closed, $k$-convex, star-shaped solutions to $r^{\a} u^{-\b}\s_k(\kappa) =\gamma$ are spheres. When $k=1$, $\a\geq\b+1$, $\b>0$, the smooth, closed, mean-convex, star-shaped hypersurfaces in $\mathbb \R^{n+1}$ $(n\geq2)$ that satisfy $r^{\a}u^{-\b}H=\gamma$ are spheres. When $n \geq 2$, $\a\geq\b+k$, $\b>0$, the smooth, closed, uniformly convex, star-shaped hypersurfaces that satisfy $r^{\a} u^{-\b}\s_k(\kappa) =\gamma$ are spheres.
\end{prop}
When $\a<\b+k$, 
using Li-Sheng-Wang's technique in \cite{LSW20a} or \cite{LSW20b}, we find that the hypersurface evolving by \eqref{s1:flow-n} does not converge to a sphere centered at the origin in general. The following counter example was first given in the proof of \cite[Theorem 1.5]{LSW20a}.
\begin{prop}\label{ce} 
Suppose $\a<\b+k$, there exists a smooth, closed, uniformly convex hypersurface ${M}_0$ such that under the flow \eqref{s1:flow-n},
\begin{align}\label{ace}
\mathcal{R}\f X\f \cdot, t \r \r:=\frac{\max_{\mathbb{S}^n} r(\cdot,t)}{\min_{\mathbb{S}^n} r(\cdot,t)}\rightarrow \infty \quad {\rm as} \quad  t\rightarrow T
\end{align}
for some $T>0$.
\end{prop}

The main difficulty to verify the long time existence of the flow \eqref{flow-s} is the $C^2$ estimate. 
In this paper, we propose two new auxiliary functions to deal with 
 flow \eqref{s1:flow-n} for $k$-convex star-shaped hypersurfaces, which we hope might be a useful tool to study other problems. In addition, contracting flows have been employed by several authors in proving geometric inequalities \cite{And96, BHL20, Sch08, Top98}. We hope that we can find new geometric inequalities by use of the flow \eqref{flow-s}.

This paper is organized as follows. In Section \ref{sec:2}, we collect some notations and conventions of star-shaped hypersurfaces, derive the evolution equations of the flow \eqref{s1:flow-n} and give some properties of elementary symmetric functions, we also show that the normalized flow \eqref{s1:flow-n} can be reduced to a parabolic PDE of the radial function. In Section \ref{sec:3}, we prove $C^0$, $C^1$ estimates and show the hypersurface preserves star-shaped along the normalized flow \eqref{s1:flow-n}. In Section \ref{sec:4}, after a delicate calculation, we obtain the $C^2$ estimates of various cases, then we have the long time existence of the flow \eqref{s1:flow-n}. In Section \ref{sec:5},
we study the asymptotic behavior of the flow \eqref{s1:flow-n} and give the proofs of Theorem \ref{main-thm-3}--\ref{main-thm-4}. In Section \ref{sec:6}, we give a counter example to verify that the condition $\a \geq \b+k$ in Theorem \ref{main-thm-3}--\ref{main-thm-4} is necessary.
\begin{ack}
	The authors would like to thank Dr. Yingxiang Hu and Dr. Yong Wei for helpful discussions. The work was supported by NSFC Grant No. 11831005 and NSFC-FWO Grant No. 11961131001.
\end{ack}

\section{Preliminaries}\label{sec:2}
In this section, we first derive the evolution equations of the normalized flow \eqref{s1:flow-n}.
\subsection{Evolution equation of the flow}
 For convenience, we rewrite \eqref{s1:flow-n} as
\begin{equation}\label{s2:flow-BGL}
\frac{\partial}{\partial t}X(x,t)= -\Phi \nu(x,t)+\gamma   X,
\end{equation}
where $\Phi=r^{\frac{\a}{\b}}F$, $F=\s_k^{\frac{1}{\b}}$ and $\gamma = (C_n^k)^{\frac{1}{\b}}$.

Let $f(\k)$ be a symmetric function of the principal curvatures $\k = (\k_1, \k_2, \cdots, \k_n)$. There exists a function $\mathcal{F}(A)$, where $A =\lbrace h_i{}^j\rbrace$, such that $F (A)=f(\k)$.
Since $h_i{}^j = \sum_l h_{il} g^{lj}$, $\mathcal{F}$ can be viewed as a function $\hat{\mathcal{F}}(h_{ij}, g_{ij})$ defined on the second fundamental form $\lbrace h_{ij}\rbrace$ and the metric $\lbrace g_{ij}\rbrace$. We denote
\begin{align*}
\dot{ \mathcal{F} }^{pq}(A): =\frac{\partial \hat{\mathcal{F}}}{\partial h_{pq}}(h_{ij}, g_{ij}) \ , \quad  \ddot{\mathcal{F}}^{pq,rs}(A): =\frac{\partial ^2 \hat{\mathcal{F}} }{\partial h_{pq}\partial h_{rs}}(h_{ij}, g_{ij}).
\end{align*}
If we diagonalize $A$ and $\lbrace g_{ij}\rbrace$ at one point, we have
\begin{align*}
\dot{f}^p(\k) :=\frac{\partial f }{\partial \kappa_p}(\k)=\dot{\mathcal{F}}^{pp}(A) \ , \quad  \ddot{f}^{pq}(\k):=\frac{\partial ^2 f}{\partial \kappa_p\partial \kappa_q}=\ddot{\mathcal{F}}^{pp,qq}(A).
\end{align*}
 We refer to \cite{And07} for detailed properties of
$\mathcal{F}(A)$.


	
%
\begin{lem}
	Along the flow \eqref{s1:flow-n}, we have the following evolution equations	(also see \cite{And07, LSW20b, Xiao20})
	\begin{align}
 \frac{\partial}{\partial t}g_{ij}=&-2\Phi h_{ij}+2 \gamma  g_{ij}, \quad
 \frac{\partial }{\partial t} \nu = \nabla \Phi \label{g} , \\
\frac{\partial}{\partial t}u=&-\Phi+\gamma u+r^{\frac{\a}{\b}} \langle X,\nabla F \rangle +\frac{\a}{\b}\Phi \langle X,\nabla \log r \rangle,\label{u} \\
\frac{\partial}{\partial t} h_i{}^j=&r^{\frac{\a}{\b}}\dot{F}^{pq}\nabla_p\nabla_qh_i{}^j-\f\frac{k}{\b}-1\r\Phi h_i{}^lh_l{}^j+r^{\frac{\a}{\b}}\dot{F}^{pq}h_p{}^lh_{lq}h_i{}^j\nonumber\\
&+r^{\frac{\a}{\b}}\ddot{F}^{pq,rs}\nabla_ih_{pq}\nabla^jh_{rs}
+\frac{\a}{\b}r^{\frac{\a}{\b}-1}\nabla_iF\nabla^jr+\frac{\a}{\b}r^{\frac{\a}{\b}-1}\nabla^jF\nabla_ir \nonumber\\
&+\frac{\a}{\b}r^{\frac{\a}{\b}-1}F\nabla_i\nabla^jr+\frac{\a}{\b}\f\frac{\a}{\b}-1\r r^{\frac{\a}{\b}-2}F\nabla_ir\nabla^jr-\gamma h_i{}^j,\label{evolution h}
\end{align}
where $\nabla$ is the Levi-Civita connection of the induced metric on $M_t$.	
\end{lem}
\begin{proof}
By direct calculations, we have
\begin{align*}
\frac{\partial}{\partial t}g_{ij}=&{\partial_t} \metric{\partial_i X}{\partial_j X}\\
=&\langle \partial_i\f-\Phi\nu+\gamma X\r,\partial_jX \rangle +\langle \partial_iX,\partial_j\f-\Phi\nu+\gamma X\r \rangle \\
=&-\Phi\f \langle \partial_i\nu,\partial_jX \rangle +\langle \partial_i X,\partial_j\nu \rangle \r+2\gamma g_{ij}\\
=&-2\Phi h_{ij}+2\gamma g_{ij},
\end{align*}
\begin{equation} \label{ev-nu}
\begin{aligned}
\frac{\partial}{\partial t} \nu=& \langle \partial_t\nu,\partial_jX \rangle g^{il}\partial_l X\\
=&-\langle \nu,\partial_j\f-\Phi\nu+\gamma X\r \rangle g^{jl}\partial_l X\\
=& \partial_j \Phi g^{jl}\partial_l X \\
=&\nabla\Phi.
\end{aligned}
\end{equation}
Using \eqref{ev-nu}, we have the evolution of the support function $u$
\begin{align*}
\frac{\partial}{\partial t}u=&\partial_t\langle X,\nu \rangle \\
=&\langle -\Phi\nu+\gamma X,\nu \rangle+\langle X,\nabla\Phi \rangle \\
=&-\Phi+\gamma u+r^{\frac{\a}{\b}} \langle X,\nabla F \rangle +\frac{\a}{\b}\Phi \langle X,\nabla \log r \rangle.
\end{align*}
Now we calculate the evolution of $h_{ij}$,
\begin{align*}
\frac{\partial}{\partial t} h_{ij}=&-\partial_t \langle \partial_i\partial_j X,\nu \rangle \\
=&- \langle \partial_i\partial_j\f-\Phi\nu+\gamma X\r,\nu \rangle \\
=&\nabla_j\nabla_i\Phi-\Phi h_i{}^l h_{lj}+\gamma h_{ij}.
\end{align*}
From \eqref{g}, we have
\begin{align*}
\frac{\partial}{\partial t} g^{ij}=-g^{il}\f \partial_t g_{lm} \r g^{mj}=2\Phi h^{ij}-2\gamma g^{ij}.
\end{align*}
Thus,
\begin{equation}
\begin{aligned}\label{h_i^j}
\frac{\partial}{\partial t} h_i{}^j=&\partial_t h_{il}g^{lj}+h_{il}\partial_t g^{lj}\\
=&\nabla^j\nabla_i\Phi+\Phi h_i{}^l h_l{}^j-\gamma h_i{}^j.
\end{aligned}
\end{equation}
For the first term of \eqref{h_i^j}, we have
\begin{equation}
\begin{aligned}\label{Phi_ij}
\nabla^j\nabla_i\Phi=&r^{\frac{\a}{\b}}\nabla^j\nabla_iF+\nabla^jr^{\frac{\a}{\b}}\nabla_iF+\nabla^jF\nabla_ir^{\frac{\a}{\b}}+F\nabla^j\nabla_ir^{\frac{\a}{\b}}\\
=&r^{\frac{\a}{\b}}\dot{F}^{pq}\nabla^j\nabla_ih_{pq} + r^{\frac{\a}{\b}}\ddot{F}^{pq,rs}\nabla^jh_{rs}\nabla_ih_{pq}+\frac{\a}{\b}r^{\frac{\a}{\b}-1}\nabla^jr\nabla_iF\\
&+ \frac{\a}{\b}r^{\frac{\a}{\b}-1}\nabla^jF\nabla_ir+\frac{\a}{\b}r^{\frac{\a}{\b}-1}F\nabla^j\nabla_ir+\frac{\a}{\b}\f\frac{\a}{\b}-1\r r^{\frac{\a}{\b}-2}F\nabla^jr\nabla_ir.
\end{aligned}
\end{equation}
Using the Codazzi equation, the Ricci identities and $\dot{F}^{pq}h_q{}^l=\dot{F}^{lq}h_q{}^p$, we have
\begin{equation}
\begin{aligned}\label{Ric id}
\dot{F}^{pq}\nabla^j\nabla_ih_{pq}=&\dot{F}^{pq}\nabla^j\nabla_qh_{pi}\\
=&\dot{F}^{pq} \nabla_q\nabla^jh_{pi}+\dot{F}^{pq}h_i{}^lR_{lpq}{}^j+\dot{F}^{pq}h_p{}^lR_{liq}{}^j\\
=&\dot{F}^{pq}\nabla_p\nabla_qh_i{}^j+\dot{F}^{pq}h_i{}^lh_{lq}h_p{}^j-\dot{F}^{pq}h_i{}^lh_l{}^jh_{pq}+\dot{F}^{pq}h_p{}^lh_{lq}h_i{}^j-\dot{F}^{pq}h_p{}^lh_{qi}h_l{}^j\\
=&\dot{F}^{pq}\nabla_p\nabla_qh_i{}^j-\frac{k}{\b}Fh_i{}^lh_l{}^j+\dot{F}^{pq}h_p{}^lh_{lq}h_i{}^j,
\end{aligned}
\end{equation}
where we used the fact that $\dot{F}^{pq} h_{pq} = \frac{k}{\b} F$ because $F$ is homogeneous of degree $\frac{k}{\b}$.
Inserting \eqref{Ric id} and \eqref{Phi_ij} into \eqref{h_i^j}, we have
\begin{align*}
\frac{\partial}{\partial t} h_i{}^j=&r^{\frac{\a}{\b}}\dot{F}^{pq}\nabla_p\nabla_qh_i{}^j-\f\frac{k}{\b}-1\r\Phi h_i{}^lh_l{}^j+r^{\frac{\a}{\b}}\dot{F}^{pq}h_p{}^lh_{lq}h_i{}^j\\
&+r^{\frac{\a}{\b}}\ddot{F}^{pq,rs}\nabla_ih_{pq}\nabla^jh_{rs}
+\frac{\a}{\b}r^{\frac{\a}{\b}-1}\nabla_iF\nabla^jr
+\frac{\a}{\b}r^{\frac{\a}{\b}-1}\nabla^jF\nabla_ir\\
&+\frac{\a}{\b}r^{\frac{\a}{\b}-1}F\nabla_i\nabla^jr+\frac{\a}{\b}\f\frac{\a}{\b}-1\r r^{\frac{\a}{\b}-2}F\nabla_ir\nabla^jr-\gamma h_i{}^j.
\end{align*}
\end{proof}
\subsection{Evolution of the radial graph}
	For a star-shaped hypersurface $M \subset \mathbb{R}^{n+1}$, we can parameterize it as a graph of the radial function $r(\theta): \mathbb{S}^n\to \mathbb{R}$, i.e.
	\begin{align*}
	M = \lbrace r(\theta) \theta : \theta \in \mathbb{S}^n \rbrace,
	\end{align*}
	where $\theta=(\theta^1,\cdots,\theta^n)$ is a local normal coordinate system of $\mathbb{S}^n$. 
	Let $f_i = D_i f$, $f_{ij} = D_{ij} f$, $f_{ijk} =D_{ijk} f$, where $D$ denotes the Levi-Civita connection on 
	$\mathbb{S}^n$.
	Similar to \cite[p. 839]{LSW20b}, the flow \eqref{s1:flow-n} can be written as the scalar equation 
	\begin{equation}\label{s1:flow-rn}
	\left\{\begin{aligned}
	\frac{\partial}{\partial t}r(\theta,t)=& -r^{\frac{\a}{\b} -1}\sqrt{r^2+|D r|^2}\s_k(\kappa)^{\frac{1}{\b}}+\gamma  r,\\
	r(\cdot,0)=& r_0(\cdot).
	\end{aligned}\right.
	\end{equation}
	We define $\g = \log r$
	, then
	\begin{equation}\label{varphig}
	\frac{\partial}{\partial t}\g=\frac{\partial}{\partial t} \log r= -e^{{(\frac{\a}{\b}-1})\g}\sqrt{1+|D\g|^2}\s_k(\kappa)^{\frac{1}{\b}}+\gamma.
	\end{equation}
	
 Now we define $\rho = \sqrt{1+ |D \g|^2}$. The support function $u= \langle X, \nu \rangle$, induced metric $g_{ij}$ and the second fundamental form $h_{ij}$ of hypersurface $M$ can be expressed as 
	\begin{equation}\label{rg-g}
	\begin{aligned}
	u = \frac{r}{\rho}, \quad
	g_{ij}=r^2(\d_{ij}+\g_i \g_j), \quad
	g^{ij}=\frac{1}{r^2}(\d_{ij}-\frac{\g_i \g_j}{\rho^2}),
	\end{aligned}
	\end{equation}
	\begin{equation}\label{rg-h}
	\begin{aligned}
	h_{ij} = \frac{r}{\rho}(\d_{ij} + \g_i \g_j -\g_{ij}), \quad
	h_i{}^j =\frac{1}{r \rho}(\d_{ij} - \g_{ij} + \frac{\g_i{}^l \g_l \g_j}{\rho^2}).
	\end{aligned}
	\end{equation}
\subsection{Elementary symmetric functions}\label{symm}
	Here we state some algebraic properties of the elementary symmetric function $\s_m(\k)$, $m=1,\cdots,n$, where $\kappa=(\kappa_1,\cdots,\kappa_n)$. We recall $\s_m(\k)$ defined by
	\begin{equation*}
	\s_m(\k)=\sum_{1\leq i_1<\cdots<i_m \leq n}\k_{i_1}\cdots \k_{i_m}, \quad m=1,2, \cdots,n.
	\end{equation*}
	The Garding cone $\Gamma_m^+$ is an open symmetric convex cone in $\mathbb{R}^n$ with vertex at the origin, given by
	\begin{align*}
	\Gamma_m^{+}=\lbrace\f \kappa_1,\cdots, \kappa_n\r\in\mathbb{R}^n |\s_j(\kappa)>0,  \forall j= 1,\cdots,m\rbrace.
	\end{align*}
	Clearly $\s_m(\kappa)=0$ for $\k \in\partial \bar{\Gamma}_m^+$ and
	\begin{align*}
	\Gamma_n^+\subset\cdots\subset {\Gamma}_m ^+\subset\cdots\subset\Gamma_1^+.
	\end{align*}
	In particular, $\Gamma^+ = \Gamma_n^+$ is called the positive cone,
	\begin{align*}
	\Gamma^+=\lbrace\f \kappa_1,\cdots, \kappa_n\r\in\mathbb{R}^n |\kappa_1>0,\cdots,\kappa_n>0\rbrace.
	\end{align*}
	We collect some inequalities 
	 related to the polynomial $\s_k(\kappa)$, which are needed in the following text. For $\kappa\in\Gamma_m^+$, from the symmetry of $\s_m(\k)$, we can assume that $\kappa_1 \geq \cdots \geq \kappa_n$. Define $\s_{m}\f \kappa|i_1 i_2 \cdots i_j \r$ as the $m$-th elementary symmetric function of $\kappa_1,\kappa_2,\cdots,\kappa_n$ with $\kappa_{i_1}=0,\kappa_{i_2}=0,\cdots, \kappa_{i_j}=0$. Then $\frac{\partial\s_m}{\partial\kappa_i}(\kappa)=\s_{m-1}(\kappa|i )$. At the same time, if we set $\s_0=1$ and $\s_l=0$ for $l>n$, one can verify the following properties (see \cite{GaoLM18, Lg}, \cite[pp. 183-184]{Wang09})
	\begin{enumerate}
		\item $\s_m(\kappa)=\s_{m}(\kappa|i)+\kappa_i\s_{m-1}(\kappa|i),$
		\item $\s_{m}(\kappa)\s_{m-2}(\kappa)\leq \frac{(m-1)(n-m+1)}{m(n-m+2)} \s_{m-1}^2(\kappa),$
		\item $ \f \frac{\s_m(\kappa)}{C_n^m} \r^{\frac{1}{m}}\leq  \f \frac{\s_l(\kappa)}{C_n^l} \r^{\frac{1}{l}}, 1\leq l\leq m,$
		\item $\sum_i \frac{\partial \s_m}{\partial \k_i} \k_i^2 = \s_1 \s_m -(m+1) \s_{m+1}$,
		\item $\kappa_1\sigma_{m-1}(\kappa|1)\geq \frac{m}{n}\s_m \f \kappa \r. $
	\end{enumerate}
    We can also find the following properties (see \cite{LT94}). If $\kappa\in\Gamma_m^+$,
	\begin{align*}
	\kappa_i\geq& 0,\quad 1\leq i\leq m,\\
	\s_{l}\f\kappa|i_1 i_2\cdots i_j  \r>&0, \quad  \forall \lbrace i_1 i_2\cdots i_j\rbrace\subseteq\lbrace1,2,\cdots,n\rbrace, \  l+j\leq m.
	\end{align*}
	Besides, 
    $ \f  \frac{\s_m}{\s_1}  \r^{\frac{1}{m-1}}(\k)$ and $\s_m^{\frac{1}{m}}(\k)$ are concave on $\Gamma_m^+$.  $\s_m^{\frac{1}{m}}(\k)$ is  inverse-concave on $\Gamma^+$. Here, we call $f(\k)$ inverse-concave, if
	\begin{align*}
	f_{-1} (\k_1, \k_2, \cdots, \k_n) := f(\k_1^{-1}, \k_2^{-1}, \cdots, \k_n^{-1})^{-1}
	\end{align*}
	is concave.
	We refer readers to \cite{CNS85,Tru90} for details.

\section{ $C^0$ and $C^1$ estimates}\label{sec:3}
In this section, we first give $C^0$ and $C^1$ estimates of the normalized flow \eqref{s1:flow-n}. As a corollary, we show that the hypersurface preserves star-shaped along the flow \eqref{s1:flow-n}.
\subsection{ $C^0$ estimate}\label{sec:3.1}
\begin{lem}\label{C0-est}
	Let $r(\cdot, t)$ be a positive, $k$-convex smooth solution to \eqref{s1:flow-n} on $\mathbb{S}^n \times [0,T)$. If $\a \geq k+\b$, $\b>0$, then there exists a positive constant $C$ depending only on $\max_{\mathbb{S}^n}r(\cdot,0)$ and $\min_{\mathbb{S}^n}r(\cdot,0)$ such that
	\begin{align*}
	\frac{1}{C} \leq r(\cdot, t) \leq C, \quad \forall t \in [0,T).
	\end{align*}
\end{lem}

\begin{proof}
	Now we focus on the scalar equation \eqref{varphig} of $\g$. Fix some time $t \in [0,T)$, at the maximum point $(\theta_0, t)\in \mathbb{S}^n$ of $\g(\cdot,t)$, we know $|D \g| =0$ and $\lbrace \g_{ij} \rbrace$ is negative semi-definite, which imply
	\begin{align*}
	\rho =1, \quad h_i^j = e^{-\g}(\d_{ij} -\g_{ij})\geq e^{-\g}\d_{ij}.
	\end{align*}
	If we let $\g_{\max}(t)=\max\g(\cdot,t)$, from \eqref{varphig} and $\s_k^{\frac{1}{\b}}(h_i{}^j) \geq  \gamma e^{- \frac{k}{\b} \g}$ at $(\theta_0, t)$, we have
	\begin{equation*}
	\begin{aligned}
	\frac{d \g_{\max}}{dt}(t) =& - e^{\f\frac{\a}{\b}-1\r\g_{\max}(t)}\s_k^{\frac{1}{\b}}(h_i{}^j) +\gamma \\
	 \leq& - \gamma e^{\f\frac{\a}{\b}-1-\frac{k}{\b}\r \g_{\max}(t)} + \gamma.
	\end{aligned}
	\end{equation*}
 For the same reason, let $\g_{\min}(t)=\min\g(\cdot,t)$, we have
	\begin{align}
	\frac{d \g_{\min}}{dt}(t) \geq -\gamma e^{\f\frac{\a}{\b}-1-\frac{k}{\b} \r\g_{\min}(t)} +\gamma.
	\end{align}
	
	If $\a =\b+k$, it means $\g_{\min}(0) \leq \g(t) \leq \g_{\max}(0)$ along the normalized flow. If $\a > \b+k$, it implies $\min \lbrace \g_{\min}(0), 0 \rbrace \leq \g (t) \leq \max \lbrace \g_{\max}(0), 0 \rbrace$. So we have
	\begin{align*}
	\frac{1}{C} \leq r \leq C
	\end{align*}
	for some constant $C>0$, which only depends on the initial hypersurface $M_0$.
	\end{proof}
\subsection{ $C^1$ estimate}
\begin{lem}\label{c1}
Let $r(\cdot, t)$ be a positive, smooth solution to \eqref{s1:flow-n} on $\mathbb{S}^n \times [0,T)$. If $\a \geq k+\b$, $\b>0$, we have
\begin{equation*}
|Dr|\leq C,
\end{equation*}
where $C$ only depends on $M_0$.
\end{lem}
\begin{proof}
For a fixed time $t$, we assume that $\theta_0 \in \mathbb{S}^n$ is the spatial maximum point of $\frac{1}{2}|D\g|^2$. At $(\theta_0, t)$, we have by use of \eqref{varphig}
\begin{equation}
\begin{aligned}\label{Dg}
\frac{\partial}{\partial t} (\frac{1}{2}|D\g|^2)=&\g_i(\partial_t \g)_i
                              =\g_i(-e^{(\frac{\a}{\b}-1)\g}\sqrt{1+|D\g|^2}\s_k^{\frac{1}{\b}})_i \\
                              =&-(\frac{\a}{\b}-1)e^{(\frac{\a}{\b}-1)\g}\sqrt{1+|D\g|^2}\s_k^{\frac{1}{\b}}|D\g|^2-e^{(\frac{\a}{\b}-1)\g}\sqrt{1+|D\g|^2}\g_i(\s_k^{\frac{1}{\b}})_i,
\end{aligned}
\end{equation}
where we used $\rho_i=(\sqrt{1+|D\g|^2}),_i=0$ at $(\theta_0, t)$. For the last term of \eqref{Dg}, we know
\begin{equation}\label{s_k,i}
\begin{aligned}
(\s_k^{\frac{1}{\b}})_i=&\frac{1}{\b}\s_k^{\frac{1}{\b}-1}\frac{\partial \s_k}{\partial h_p{}^q}h_p{}^q{},_i
                       =\frac{1}{\b}\s_k^{\frac{1}{\b}-1}\s_k^{pq}g_{qs}h_p{}^s{},_i\\
                       =&\frac{1}{\b}\s_k^{\frac{1}{\b}-1}\s_k^{pq}\f(g_{qs}h_p{}^s),_i-h_p{}^sg_{qs,i} \r\\          =&\frac{1}{\b}\s_k^{\frac{1}{\b}-1}\s_k^{pq}(h_{pq,i}-h_p{}^sg_{qs,i}),
\end{aligned}
\end{equation}
where we regarded $h_{pq}$, $g_{qs}$ as tensors on $\mathbb S^n$.

Now we calculate the term $h_{pq,i}$. Differentiating $\g_i=\frac{r_i}{r}$, we have $\g_{ij}=\frac{r_{ij}}{r}-\frac{r_ir_j}{r^2}$, which gives $\frac{r_{ij}}{r}=\g_{ij}+\g_i\g_j$.
By \eqref{rg-h}, we have at $(\theta_0, t)$
\begin{align}\label{hpqi}
h_{pq,i}=\frac{r}{\rho}\g_i(\d_{pq}+\g_p\g_q-\g_{pq})+\frac{r}{\rho}(\g_{pi}\g_q+\g_p\g_{qi}-\g_{pqi}).
\end{align}
From \eqref{rg-g}, \eqref{s_k,i} and \eqref{hpqi}, we have at $(\theta_0, t)$
\begin{equation}\label{phi i s k i}
\begin{aligned}
\g_i(\s_k{}^{\frac{1}{\b}})_i=&\frac{1}{\b}\s_k{}^{(\frac{1}{\b}-1)}\s_k{}^{pq}h_{pq,i}\g_i-\frac{1}{\b}\s_k{}^{(\frac{1}{\b}-1)}(\s_k)^{pq}h_p{}^sg_{qs,i}\g_i\\
                             =&\frac{1}{\b}\s_k{}^{(\frac{1}{\b}-1)}\s_k{}^{pq}\frac{r}{\rho} \f |D\g|^2(\d_{pq}+\g_p\g_q-\g_{pq})+\g_i\g_q\g_{pi}+\g_i\g_p\g_{qi}-\g_i\g_{pqi} \r \\
                             &- \frac{1}{\b}\s_k{}^{(\frac{1}{\b}-1)}\s_k^{pq}h_p{}^s\g_i \f 2e^{2\g}(\d_{qs}+\g_q\g_s)\g_i +e^{2\g}\g_{qi}\g_s+e^{2\g}\g_q\g_{si} \r \\
                             =&\frac{1}{\b}\frac{r}{\rho}\s_k{}^{(\frac{1}{\b}-1)}\s_k{}^{pq}(\d_{pq}+\g_p\g_q-\g_{pq})|D\g|^2-\frac{1}{\b}\frac{r}{\rho}\s_k{}^{(\frac{1}{\b}-1)}\s_k{}^{pq}\g_i\g_{pqi}\\
                             &- \frac{2}{\b}e^{2\g}\s_k{}^{(\frac{1}{\b}-1)}\s_k^{pq}h_p{}^s(\d_{qs}+\g_q\g_s)|D\g|^2,
\end{aligned}
\end{equation}
where we used $\frac{1}{2}D_p|D\g|^2=\g_i\g_{pi}=0$. Inserting \eqref{rg-g} and \eqref{rg-h} into \eqref{phi i s k i}, we have
\begin{equation}\label{phi i s k i 2}
\begin{aligned}
\g_i(\s_k{}^{\frac{1}{\b}})_i=&\frac{1}{\b}\s_k{}^{(\frac{1}{\b}-1)}k\s_k|D\g|^2-\frac{1}{\b}\frac{e^{\g}}{\rho}\s_k{}^{(\frac{1}{\b}-1)}\s_k{}^{pq}\g_i\g_{pqi}-\frac{2}{\b}\s_k{}^{(\frac{1}{\b}-1)}k\s_k|D\g|^2\\
                             =&-\frac{k}{\b}\s_k{}^{\frac{1}{\b}}|D\g|^2-\frac{1}{\b}\frac{e^{\g}}{\rho}\s_k{}^{(\frac{1}{\b}-1)}\s_k{}^{pq}\g_i\g_{pqi}.
\end{aligned}
\end{equation}
By the Ricci identity
\begin{align*}
\g_{pqi}=&\g_{piq}+\g_lR^{\mathbb{S}^n}_{lpqi}\nonumber\\
         =&\g_{ipq}+\g_l(\d_{lq}\d_{pi}-\d_{li}\d_{pq}\nonumber)\\
         =&\g_{ipq}+\g_q\d_{pi}-\g_i\d_{pq}.
\end{align*}
Thus
\begin{equation} \label{phi i phi pqi}
\g_i\g_{pqi}=\g_i\g_{ipq}+\g_p\g_q-\d_{pq}|D\g|^2=(\frac{1}{2}|D\g|^2)_{pq}-\g_{ip}\g_{iq}+\g_p\g_q-\d_{pq}|D\g|^2.
\end{equation}
Putting \eqref{phi i phi pqi} into \eqref{phi i s k i 2}, we have
\begin{equation}\label{phi i s k i 3}
\g_i(\s_k{}^{\frac{1}{\b}})_i=-\frac{k}{\b}\s_k{}^{\frac{1}{\b}}|D\g|^2-\frac{1}{\b}\frac{e^{\g}}{\rho}\s_k{}^{(\frac{1}{\b}-1)}  \s_k{}^{pq}\f (\frac{1}{2}|D\g|^2)_{pq}-\g_{ip}\g_{iq}+\g_p\g_q- \d_{pq}|D\g|^2 \r .
\end{equation}
Plugging \eqref{phi i s k i 3} into \eqref{Dg}, we have
\begin{align}
\frac{\partial }{\partial t}(\frac{1}{2}|D\g|^2)=& -(\frac{\a}{\b}-1)e^{(\frac{\a}{\b}-1)\g}\rho\s_k^{\frac{1}{\b}}|D\g|^2+e^{(\frac{\a}{\b}-1)\g}\rho\frac{k}{\b}\s_k^{\frac{1}{\b}}|D\g|^2\nonumber\\
                              &+ \frac{1}{\b}e^{(\frac{\a}{\b}-1)\g}\rho\s_k^{(\frac{1}{\b}-1)}\frac{e^{\g}}{\rho} \f \s_k{}^{pq}(\frac{1}{2}|D\g|^2)_{pq}-\s_k{}^{pq}\g_{ip}\g_{iq}+\s_k{}^{pq}\g_p\g_q-\s_k{}^{pq} \d_{pq}|D\g|^2 \r \nonumber\\
                              =& -\frac{(\a-\b-k)}{\b}e^{(\frac{\a}{\b}-1)\g}\rho\s_k{}^{\frac{1}{\b}}|D\g|^2 \nonumber \\
                              &+\frac{1}{\b}e^{\frac{\a}{\b}\g}\s_k{}^{(\frac{1}{\b}-1)} \f \s_k{}^{pq}(\frac{1}{2}|D\g|^2)_{pq}
                              -\s_k{}^{pq}\g_{ip}\g_{iq}+\s_k{}^{pq}\g_p\g_q-\s_k{}^{pq} \d_{pq}|D\g|^2 \r \nonumber\\
                              \leq& -\frac{(\a-\b-k)}{\b}e^{(\frac{\a}{\b}-1)\g}\rho\s_k{}^{\frac{1}{\b}}|D\g|^2-\frac{1}{\b}e^{\frac{\a}{\b}\g}\s_k{}^{(\frac{1}{\b}-1)}\s_k{}^{pq}\g_{ip}\g_{iq} \nonumber\\
                              &-\frac{1}{\b}e^{\frac{\a}{\b}\g}\s_k{}^{(\frac{1}{\b}-1)}\s_k{}^{pq}(\d_{pq}|D\g|^2-\g_p\g_q), \label{C1-final-ineq}
\end{align}
where the last inequality was from $(|D\g|^2)_{pq} \leq 0$ at $(\theta_0,t)$. When $\a\geq\b+k$, from $\s_k{}^{pq}\geq 0$ and $\d_{pq}|D\g|^2-\g_p\g_q\geq 0$, we get $\partial_ t(\frac{1}{2}|D\g|^2)\leq 0$ at $ (\theta_0, t)$. So $|D \g|$ is bounded from above by a constant. By Lemma \ref{C0-est}, we have $|Dr|\leq C$ for some constant $C$.
\end{proof}
\begin{cor}\label{cor u}
	Under the same assumptions in Lemma \ref{C0-est}, along the normalized flow \eqref{s1:flow-n}, the hypersurface $M_t$ preserves star-shaped and the support function $u$ satisfies
	\begin{align*}
	\frac{1}{C} \leq u \leq C
	\end{align*}
	for some constant $C >0$, where $C$ only depends on the initial hypersurface $M_0$.
\end{cor}
\begin{proof}
	As $|D \g|$ and $r$ are bounded by two positive constants, it follows that
	\begin{align*}
	u =\frac{r}{\rho} = \frac{r}{\sqrt{1 +|D \g|^2}}
	\end{align*}
	is also bounded, i.e., $\frac{1}{C} \leq u \leq C$ for some $C$. Since $u = \metric{X}{\nu} = r \metric{\partial_r}{\nu}$, we have
	\begin{align*}
	\metric{\partial_r}{\nu} = \frac{u}{r} =\frac{1}{\rho}= (1+ |D \g|^2)^{-\frac{1}{2}} \geq \frac{1}{C'}
	\end{align*}
	for some $C'>0$. Thus, the hypersurface preserves star-shaped along the flow \eqref{s1:flow-n}.
\end{proof}
\section{$C^2$ estimates}\label{sec:4}
In this section, we will derive the $C^2$ estimates of flow \eqref{s1:flow-n}. To simplify the statements of  following lemmas, we always assume $\a \geq \b+k$ and the initial hypersurface is smooth, closed, star-shaped and at least $k$-convex. Besides, we define a parabolic operator
\begin{align*}
\mathcal{L} = \partial_t -r^{\frac{\a}{\b}} \dot{F}^{ij} \nabla_j \nabla_i.
\end{align*}
\subsection{The bounds of $F = \sigma_k^{\frac{1}{\b}}$}
\begin{lem}\label{F}
Along the normalized flow \eqref{s1:flow-n}, there exists a constant $C>0$ such that
\begin{equation*}
F\geq C.
\end{equation*}
\end{lem}
\begin{proof}
First, we calculate the evolution equation of $\Phi$. From \eqref{h_i^j}, we have
\begin{equation}
\begin{aligned}
\frac{\partial}{\partial t}\Phi=&r^{\frac{\a}{\b}} \partial_t F+F\partial_t \f r^{\frac{\a}{\b}} \r \\
              =&r^{\frac{\a}{\b}}
              \frac{\partial F}{\partial h_i{}^j}\partial_t h_i{}^j+\frac{\a}{2\b}r^{\frac{\a}{\b}-2} F \partial_t \langle X, X \rangle \\
              =&r^{\frac{\a}{\b}}\frac{\partial F}{\partial h_i{}^j}(\nabla^j \nabla_i \Phi+\Phi h_i{}^l h_l{}^j-\gamma  h_i{}^j)+\frac{\a}{\b}r^{\frac{\a}{\b}-2}F(-\Phi u+\gamma  r^2)\\
              =&r^{\frac{\a}{\b}}\dot{F}^{ij}\nabla_j \nabla_i \Phi+r^{\frac{\a}{\b}}\Phi\dot{F}^{ij}h_i{}^l h_{lj}-\frac{\a}{\b}\frac{\Phi^2}{r^2}u+\frac{(\a-k)\gamma}{\b} \Phi.\label{phi}
\end{aligned}
\end{equation}
 At the spatial minimum point of $\Phi$ on $M_t$, we have $r^{\frac{\a}{\b}}\dot{F}^{ij}\nabla_j \nabla_i\Phi\geq 0$ 
and $r^{\frac{\a}{\b}}\Phi\dot{F}^{ij}h_i{}^l h_l{}^j=r^{\frac{\a}{\b}}\Phi\dot{f}^i\kappa_i{}^2\geq 0$.
Then, we have
\begin{equation*}
\begin{aligned}
	\frac{d}{dt} \Phi_{\min} (t) \geq&- \frac{\a}{\b}\frac{\Phi_{\min}^2(t)}{r^2} u+ \frac{(\a-k)\gamma}{\b} \Phi_{\min}(t) \\
	 \geq& -c_1 \Phi_{\min}^2(t) + 
	 \gamma \Phi_{\min}(t),
\end{aligned}
\end{equation*}
for some constant $c_1 >0$, where $c_1$ can be given by Lemma \ref{C0-est} and Corollary \ref{cor u}. If $0 \leq \Phi_{\min}(t) \leq \frac{\gamma}{c_1}$, we have $\frac{d}{dt} \Phi_{\min}(t) \geq 0$. Therefore $\Phi_{\min} (t) \geq \min \lbrace \Phi_{\min} (0), \frac{\gamma}{c_1} \rbrace$. Since $F = r^{- \frac{\a}{\b}}\Phi$ and $r$ is bounded, $F$ is bounded from below by a constant $C$ that only depends on $M_0$, $\a$, $\b$.
\end{proof}

\begin{lem}\label{F2}
Along the normalized flow \eqref{s1:flow-n}, there exists a constant $C>0$, such that
\begin{equation*}
F\leq C.
\end{equation*}
\end{lem}
\begin{proof}
Let $Q= \log\Phi-\log(u-a)$, where $a = \frac{1}{2}\inf_{M \times [0,T)} u$. 
Now we calculate the evolution of $u$,
	\begin{align*}
	\nabla_i u =& \metric{X}{\partial_i \nu} = h_{i}{}^l \metric{X}{\partial_l X}, \\
	\nabla_j \nabla_i u =& \nabla^l h_{ij}\metric{X}{\partial_l X} + h_{ij} - h_i{}^lh_{lj} u,
	\end{align*}
where we used the Codazzi equation in the second equality. From \eqref{u}, we have
\begin{equation}
\begin{aligned}\label{Lu}
\mathcal{L}u=&\partial_t u-
r^{\frac{\a}{\b}}\dot{F}^{ij}\nabla_i \nabla_ju\\
=&-\f 1+\frac{k}{\b}\r \Phi+\gamma u +\frac{\a}{\b}\Phi \metric{X}{\nabla \log r}+r^{\frac{\a}{\b}}
\dot{F}^{ij}h_i{}^lh_{lj} u.
\end{aligned}
\end{equation}
For a fixed time $t$, we calculate at the maximum point of $Q$ on $M_t$, where we have
\begin{align}\label{cr-Q}
\frac{\nabla \Phi}{\Phi} = \frac{\nabla u}{u-a}.
\end{align}
From \eqref{phi}, \eqref{Lu} and \eqref{cr-Q},  
we have at the maximum point of $Q$
\begin{align*}
\mathcal{L}Q=&\frac{\mathcal{L}\Phi}{\Phi}-\frac{\mathcal{L}u}{u-a}\\
  =&r^{\frac{\a}{\b}}\dot{F}^{ij}h_i{}^l h_{lj}-\frac{\a}{\b}\frac{\Phi}{r^2}u+\frac{(\a-k)\gamma}{\b}+(1+\frac{k}{\b})\frac{\Phi}{u-a}-\gamma  \frac{u}{u-a}\\
  &-\frac{\a}{\b}\frac{\Phi}{u-a}\langle X,\nabla \log r\rangle -\frac{u}{u-a}r^{\frac{\a}{\b}} \dot{F}^{ij}h_i{}^lh_{lj}\\
  \leq&-\frac{a}{u-a}r^{\frac{\a}{\b}}\dot{F}^{ij}h_i{}^lh_{lj}+CF+C
\end{align*}
for some $C>0$ by 
 $C^0$, $C^1$ estimates. 
Then $\dot{F}^{ij}h_i{}^l h_l{}^j$ can be written as $\sum_i \dot{f}^i\kappa_i{}^2$.
By the formula $\sum_i \frac{\partial \s_k}{\partial \k_i} \k_i^2 = \s_1 \s_k -(k+1)\s_{k+1}$ and Newton-MacLaurin inequalities, we have
\begin{equation*}
\begin{aligned}
\sum_i\dot{f}^i\kappa_i{}^2=&\sum_i \frac{\partial \s_k^{\frac{1}{\b}}}{\partial \k_i} \k_i^2 = \frac{1}{\b}\s_k^{\frac{1}{\b}-1}(H\s_k -(k+1) \s_{k+1}) \\
=& \frac{1}{\b}H \s_k^{\frac{1}{\b}} - \frac{k+1}{\b} \s_k^{\frac{1-\b}{\b}} \s_{k+1}\\
\geq& (C_n^k)^{-\frac{1}{k}} \frac{k}{\b} \s_k^{\frac{1}{\b} + \frac{1}{k}} =  (C_n^k)^{-\frac{1}{k}} \frac{k}{\b}f^{1+ \frac{\b}{k}}.
\end{aligned}
\end{equation*}
Therefore
\begin{align*}
 \mathcal{L}Q\leq&-c_1 F^{(1+ \frac{\b}{k})}+ CF+C
\end{align*}
 for some constants $c_1$, $C >0$. Thus, there exists a constant $C>0$ that only depend on $M_0$, whenever $F >C$, we have $\frac{d}{dt}Q_{\max}(t) <0$. From Lemma \ref{C0-est} and Corollary \ref{cor u}, $r$ and $u$ are bounded. Hence $F$ goes to infinity when $Q$ goes to infinity. Therefore $Q$ is bounded from above by a constant, which gives an upper bound of $F$.
\end{proof}

\begin{cor}\label{cor Phi}
	Along the flow \eqref{s1:flow-n}, $\Phi = r^{\frac{\a}{\b}} F$ is bounded, i.e.
	\begin{align*}
	\frac{1}{C} < \Phi(x, t) < C
	\end{align*}
	for some constant $C>0$ that only depends on $M_0$, $\a$, $\b$, $k$.
\end{cor}
\begin{proof}
	From Lemma \ref{F} and Lemma \ref{F2} we know $F$ is bounded. From Lemma \ref{C0-est}, $r$ is also bounded by two positive constants. Then the corollary follows.
\end{proof}

\subsection{The bound of principal curvatures}
\begin{lem}\label{Lem 4.3}
Under the flow \eqref{s1:flow-n}, when $k=1$, $\b>0$ and $\a\geq\b+1$, the principal curvatures of the mean-convex solution have a uniform bound, i.e.
\begin{equation*}
|\kappa_i|\leq C, \quad i=1,\cdots,n.
\end{equation*}
\end{lem}
\begin{proof}
In this case, $F = H^{\frac{1}{\b}}$, $\mathcal{L}=\partial_t-\frac{1}{\b}r^{\frac{\a}{\b}}H^{\frac{1}{\b}-1}\Delta$.
We observe that
\begin{align*}
\dot{F}^{pq}= \frac{1}{\b}H^{\frac{1}{\b}-1} g^{pq}, \quad
\ddot{F}^{pq,rs}= \frac{1-\b}{\b^2}H^{\frac{1}{\b}-2}g^{rs}g^{pq}.
\end{align*}
Thus,
\begin{align}\label{H''}
\ddot{F}^{pq,rs} \nabla_ih_{pq}\nabla^jh_{rs}
=-(\b-1)H^{-\frac{1}{\b}} \nabla_i H^{\frac{1}{\b}}\nabla^j H^{\frac{1}{\b}}.
\end{align}
Now we calculate the evolution of $|A|^2$. By \eqref{evolution h} and \eqref{H''}, we have
\begin{align}
\mathcal{L}|A|^2=&2h_j{}^i\mathcal{L}h_i{}^j-2\frac{1}{\b}r^{\frac{\a}{\b}}H^{\frac{1}{\b}-1}|\nabla A|^2
\nonumber\\
=&-2\f \frac{1}{\b}-1 \r \Phi h_i{}^lh_l{}^jh_j{}^i
+ 2\frac{1}{\b} r^{\frac{\a}{\b}} H^{\frac{1}{\b}-1}|A|^4
+2 r^{\frac{\a}{\b}}h_i{}^j \ddot{F}^{pq,rs} \nabla_ih_{pq}\nabla^jh_{rs}
\nonumber\\
&+ 4 \frac{\a}{\b} r^{\frac{\a}{\b}-1} h_j{}^i \nabla_i H^{\frac{1}{\b}} \nabla^j r +2\frac{\a}{\b} r^{\frac{\a}{\b}-1} H^{\frac{1}{\b}} h_j{}^i \nabla_i\nabla^jr+ 2\frac{\a}{\b} \f \frac{\a}{\b}-1 \r r^{\frac{\a}{\b}-2}H^{\frac{1}{\b}}h_j{}^i\nabla_ir\nabla^jr
\nonumber\\
&-2\gamma|A|^2-2\frac{1}{\b}r^{\frac{\a}{\b}}H^{\frac{1}{\b}-1}|\nabla A|^2
\nonumber\\
=&-2\f \frac{1}{\b}-1 \r \Phi h_i{}^lh_l{}^jh_j{}^i+ 2\frac{1}{\b} r^{\frac{\a}{\b}} H^{\frac{1}{\b}-1}|A|^4- 2\f\b-1 \r r^{\frac{\a}{\b}}H^{-\frac{1}{\b}}h_j{}^i\nabla_iH^{\frac{1}{\b}}\nabla^jH^{\frac{1}{\b}}
\nonumber\\
&+4 \frac{\a}{\b} r^{\frac{\a}{\b}-1} h_j{}^i \nabla_i H^{\frac{1}{\b}} \nabla^j r
+2\frac{\a}{\b} r^{\frac{\a}{\b}-1} H^{\frac{1}{\b}} h_j{}^i \nabla_i\nabla^jr
+ 2\frac{\a}{\b} \f \frac{\a}{\b}-1 \r r^{\frac{\a}{\b}-2}H^{\frac{1}{\b}}h_j{}^i\nabla_ir\nabla^jr
\nonumber\\
&-2\gamma|A|^2- 2\frac{1}{\b}r^{\frac{\a}{\b}}H^{\frac{1}{\b}-1}|\nabla A|^2\label{L A^2}.
\end{align}
From \eqref{phi}, we have
\begin{align}\label{L Phi}
\mathcal{L}\Phi=\frac{1}{\b}r^{\frac{\a}{\b}} H^{\frac{1}{\b}-1}\Phi|A|^2-\frac{\a}{\b}\frac{\Phi^2}{r^2}u+\frac{(\a-1) \gamma}{\b}\Phi.
\end{align}
Now we introduce a new auxiliary function
\begin{equation}\label{aux-fun-1}
Q=\log|A|^2-2B\log\f \Phi-a\r,
\end{equation}
where $B=1-\frac{a}{2c}$, $a=\frac{1}{2}\inf_{M \times [0,T)} \Phi>0$ and $c=\sup_{M \times [0,T)}\Phi>0$. The existence of $a$ and $c$ follows from Corollary \ref{cor Phi}. Then by using \eqref{L A^2} and \eqref{L Phi}, we obtain
\begin{align}\label{LQ 1}
\mathcal{L}Q=&
\frac{\mathcal{L}|A|^2}{|A|^2}
-2B\frac{\mathcal{L}\Phi}{\Phi-a}
+\frac{1}{\b}r^{\frac{\a}{\b}}H^{\frac{1}{\b}-1}\frac{|\nabla|A|^2|^2}{|A|^4}
-2B\frac{1}{\b}r^{\frac{\a}{\b}}H^{\frac{1}{\b}-1}\frac{|\nabla\Phi|^2}{\f\Phi-a\r^2}
\nonumber\\
=&-2\f \frac{1}{\b}-1 \r \Phi \frac{h_i{}^lh_l{}^jh_j{}^i}{|A|^2}
- \frac{1}{\b}(2B\frac{\Phi}{\Phi-a}-2) r^{\frac{\a}{\b}} H^{\frac{1}{\b}-1}|A|^2
- 2\f\b-1 \r r^{\frac{\a}{\b}}H^{-\frac{1}{\b}}\frac{h_j{}^i}{|A|^2}\nabla_iH^{\frac{1}{\b}}\nabla^jH^{\frac{1}{\b}}
\nonumber\\
&+4 \frac{\a}{\b} r^{\frac{\a}{\b}-1} \frac{h_j{}^i}{|A|^2} \nabla_i H^{\frac{1}{\b}} \nabla^jr
+2\frac{\a}{\b} r^{\frac{\a}{\b}-1} H^{\frac{1}{\b}} \frac{h_j{}^i}{|A|^2} \nabla_i\nabla^jr
+ 2\frac{\a}{\b} \f \frac{\a}{\b}-1 \r r^{\frac{\a}{\b}-2}H^{\frac{1}{\b}}\frac{h_j{}^i}{|A|^2}\nabla_ir\nabla^jr
\nonumber\\
&-2\gamma-2\frac{1}{\b}r^{\frac{\a}{\b}}H^{\frac{1}{\b}-1}\frac{|\nabla A|^2}{|A|^2}
+2B\frac{\a}{\b}\frac{\Phi^2u}{r^2(\Phi-a)}
-2B\frac{(\a-1)\gamma}{\b}\frac{\Phi}{\Phi-a}
\nonumber\\
&+\frac{1}{\b}r^{\frac{\a}{\b}}H^{\frac{1}{\b}-1}\frac{|\nabla|A|^2|^2}{|A|^4}
-2B\frac{1}{\b}r^{\frac{\a}{\b}}H^{\frac{1}{\b}-1}\frac{| \nabla \Phi|^2}{\f\Phi-a\r^2}.
\end{align}
By the definition of $B$, 
we have
\begin{align}\label{Bestimate}
2B\frac{\Phi}{\Phi-a}-2 =\f 2- \frac{a}{c}\r \frac{\Phi}{\Phi-a} -2
\geq \f 2- \frac{a}{c}\r \frac{c}{c-a}-2 = \frac{a}{c-a}.
\end{align}
Using an orthonormal frame, it is easy to see
\begin{align}\label{key-ineq-lem-4.3}
\frac{\metric{ \nabla |A|^2}{\nabla \Phi}}{(\Phi-a)}
= 2\nabla_l h_{ij} \f h_{ij} \frac{\nabla_l \Phi}{\Phi-a} \r
\leq  |\nabla A|^2 + |A|^2 \frac{|\nabla \Phi|^2}{(\Phi-a)^2}.
\end{align}
In \eqref{LQ 1}, we use  \eqref{Bestimate} in the second term and use \eqref{key-ineq-lem-4.3} in the eighth term, then
\begin{equation}\label{LQ 3}
\begin{aligned}
\mathcal{L}Q\leq&-2\f \frac{1}{\b}-1 \r \Phi \frac{h_i{}^lh_l{}^jh_j{}^i}{|A|^2}
- \frac{a}{c-a}\frac{1}{\b} r^{\frac{\a}{\b}} H^{\frac{1}{\b}-1}|A|^2
- 2\f\b-1 \r r^{\frac{\a}{\b}}H^{-\frac{1}{\b}}\frac{h_j{}^i}{|A|^2}\nabla_iH^{\frac{1}{\b}}\nabla^jH^{\frac{1}{\b}}
\\
&+4 \frac{\a}{\b} r^{\frac{\a}{\b}-1} \frac{h_j{}^i}{|A|^2} \nabla_i H^{\frac{1}{\b}} \nabla^jr
+2\frac{\a}{\b} r^{\frac{\a}{\b}-1} H^{\frac{1}{\b}} \frac{h_j{}^i}{|A|^2} \nabla_i\nabla^jr
+ 2\frac{\a}{\b} \f \frac{\a}{\b}-1 \r r^{\frac{\a}{\b}-2}H^{\frac{1}{\b}}\frac{h_j{}^i}{|A|^2}\nabla_ir\nabla^jr
\\
&-2\gamma+2\frac{1}{\b}r^{\frac{\a}{\b}}H^{\frac{1}{\b}-1}\frac{|\nabla\Phi|^2}{\f\Phi-a\r^2}-2\frac{1}{\b}r^{\frac{\a}{\b}}H^{\frac{1}{\b}-1}\frac{ \metric{\nabla\Phi}{\nabla|A|^2} }{|A|^2\f \Phi-a\r }
+2B\frac{\a}{\b}\frac{\Phi^2u}{r^2(\Phi-a)}\\
&-2B\frac{(\a-1)\gamma}{\b}\frac{\Phi}{\Phi-a}
+\frac{1}{\b}r^{\frac{\a}{\b}}H^{\frac{1}{\b}-1}\frac{|\nabla|A|^2|^2}{|A|^4}
-2B\frac{1}{\b}r^{\frac{\a}{\b}}H^{\frac{1}{\b}-1}\frac{| \nabla \Phi|^2}{\f\Phi-a\r^2}.
 \end{aligned}
 \end{equation}
At the maximum point p of $Q$ on $M_t$, we have the critical equation
 \begin{align}\label{critical equation}
 \frac{\nabla|A|^2}{|A|^2}-2B\frac{\nabla\Phi}{\Phi-a}=0.
 \end{align}
Plugging \eqref{critical equation} and $B=1-\frac{a}{2c}$ into \eqref{LQ 3}, we obtain
\begin{align*}
\mathcal{L} Q
\leq&-2\f \frac{1}{\b}-1 \r \Phi \frac{h_i{}^lh_l{}^jh_j{}^i}{|A|^2}
- \frac{a}{c-a}\frac{1}{\b} r^{\frac{\a}{\b}} H^{\frac{1}{\b}-1}|A|^2
- 2\f\b-1 \r r^{\frac{\a}{\b}}H^{-\frac{1}{\b}}\frac{h_j{}^i}{|A|^2}\nabla_iH^{\frac{1}{\b}}\nabla^jH^{\frac{1}{\b}}
\\
&+4 \frac{\a}{\b} r^{\frac{\a}{\b}-1} \frac{h_j{}^i}{|A|^2} \nabla_i H^{\frac{1}{\b}} \nabla^jr
+2\frac{\a}{\b} r^{\frac{\a}{\b}-1} H^{\frac{1}{\b}} \frac{h_j{}^i}{|A|^2} \nabla_i\nabla^jr
+ 2\frac{\a}{\b} \f \frac{\a}{\b}-1 \r r^{\frac{\a}{\b}-2}H^{\frac{1}{\b}}\frac{h_j{}^i}{|A|^2}\nabla_ir\nabla^jr
\\
&-2\gamma- \frac{a}{c}\f 1-\frac{a}{c}\r\frac{1}{\b}r^{\frac{\a}{\b}}H^{\frac{1}{\b}-1}\frac{|\nabla\Phi|^2}{\f\Phi-a\r^2}
+\f 2-\frac{a}{c} \r \frac{\a}{\b}\frac{\Phi^2u}{r^2(\Phi-a)}
-\f 2-\frac{a}{c} \r \frac{(\a-1)\gamma}{\b}\frac{\Phi}{\Phi-a}.
\end{align*}
Diagonalizing the second fundamental form at p with an orthonormal frame $\{ e_1, \cdots, e_n\}$, we have
\begin{align}\label{LQ 2}
\mathcal{L} Q
\leq&-2\f \frac{1}{\b}-1 \r \Phi \frac{\sum_i\kappa_i{}^3}{|A|^2}
- \frac{a}{c-a}\frac{1}{\b} r^{\frac{\a}{\b}} H^{\frac{1}{\b}-1}|A|^2
- 2\f\b-1 \r r^{\frac{\a}{\b}}H^{-\frac{1}{\b}}\sum_i\frac{\kappa_i}{|A|^2}\f\nabla_iH^{\frac{1}{\b}}\r^2
\nonumber\\
&+ 4\frac{\a}{\b} r^{\frac{\a}{\b}-1} \sum_i\frac{\kappa_i}{|A|^2} \nabla_i H^{\frac{1}{\b}} \nabla_i r
+2\frac{\a}{\b} r^{\frac{\a}{\b}-1} H^{\frac{1}{\b}}\sum_i \frac{\kappa_i}{|A|^2} \nabla_i\nabla_ir
+ 2\frac{\a}{\b} \f \frac{\a}{\b}-1 \r r^{\frac{\a}{\b}-2}H^{\frac{1}{\b}}\sum_i\frac{\kappa_i}{|A|^2}\f\nabla_ir\r^2
\nonumber\\
&-2\gamma- \frac{a}{c}\f 1-\frac{a}{c}\r\frac{1}{\b}r^{\frac{\a}{\b}}H^{\frac{1}{\b}-1}\frac{|\nabla\Phi|^2}{\f\Phi-a\r^2}
+\f 2-\frac{a}{c} \r \frac{\a}{\b}\frac{\Phi^2u}{r^2(\Phi-a)}
-\f 2-\frac{a}{c} \r \frac{(\a-1)\gamma}{\b}\frac{\Phi}{\Phi-a}.
\end{align}
Denote $|\kappa|_{\max}=\max\lbrace|\kappa_1|,\cdots,|\kappa_n|\rbrace$, then
\begin{align}
\frac{|\kappa_i|}{|A|^2}\leq& \frac{1}{|\k|_{\max}}\leq\frac{\sqrt{n}}{|A|}, \quad \forall \ i=1, \cdots,n, \label{k_i est}\\
\frac{|\sum_i\kappa_i{}^3|}{|A|^2}\leq &\frac{n|\k|_{\max}^3}{|\k|_{\max}^2} = n|\k|_{\max}\leq n|A|. \label{k_i^3 est}
\end{align}
By using \eqref{k_i est} and \eqref{k_i^3 est}, we can estimate the first line of \eqref{LQ 2} as
\begin{equation}
\begin{aligned} \label{LQ 1l}
&-2\f \frac{1}{\b}-1 \r \Phi \frac{\sum_i\kappa_i{}^3}{|A|^2}
- \frac{a}{c-a}\frac{1}{\b} r^{\frac{\a}{\b}} H^{\frac{1}{\b}-1}|A|^2
- 2\f\b-1 \r r^{\frac{\a}{\b}}H^{-\frac{1}{\b}}\sum_i\frac{\kappa_i}{|A|^2}\f\nabla_iH^{\frac{1}{\b}}\r^2
\\
\leq& -c_1 |A|^2 +C|A| +C \frac{|\nabla H^{\frac{1}{\b}}|^2}{|A|}
\end{aligned}
\end{equation}
for some $C$, $c_1>0$. By direct computation, we also obtain
\begin{align}
\nabla_i r=&\frac{\nabla_i \metric{X}{X}}{2r}=\frac{ \metric{X}{e_i}}{r}\leq 1, \quad |\nabla r| \leq 1, \label{ri}\\
\nabla_j \nabla_i r=&\frac{\nabla_j \metric{X}{e_i}}{r}-\frac{\nabla_i r\nabla_j r}{r}=\frac{\delta_{ij}}{r}-\frac{u}{r}h_{ij}-\frac{\nabla_i r\nabla_j r}{r}. \label{rij}
\end{align}
Thus the second line of \eqref{LQ 2} becomes
\begin{align}\label{LQ 2l}
&4\frac{\a}{\b} r^{\frac{\a}{\b}-1} \sum_i\frac{\kappa_i}{|A|^2} \nabla_i H^{\frac{1}{\b}} \nabla_i r
+2\frac{\a}{\b} r^{\frac{\a}{\b}-1} H^{\frac{1}{\b}}\sum_i \frac{\kappa_i}{|A|^2} \nabla_i\nabla_ir
+ 2\frac{\a}{\b} \f \frac{\a}{\b}-1 \r r^{\frac{\a}{\b}-2}H^{\frac{1}{\b}}\sum_i\frac{\kappa_i}{|A|^2}\f\nabla_ir\r^2\nonumber\\
=& 4\frac{\a}{\b} r^{\frac{\a}{\b}-1} \sum_i\frac{\kappa_i}{|A|^2} \nabla_i H^{\frac{1}{\b}} \nabla_i r +2\frac{\a}{\b} r^{\frac{\a}{\b}-2} H^{\frac{1}{\b}} \sum_i \frac{\kappa_i}{|A|^2}-2\frac{\a}{\b} r^{\frac{\a}{\b}-2}uH^{\frac{1}{\b}}\nonumber\\
&+ 2\frac{\a}{\b} \f \frac{\a}{\b}-2 \r r^{\frac{\a}{\b}-2}H^{\frac{1}{\b}}\sum_i \frac{\kappa_i}{|A|^2}\f\nabla_i r \r ^2 \nonumber\\
\leq& C+ C\frac{1}{|A|}+C \frac{|\nabla H^{\frac{1}{\b}}|}{|A|}
\end{align}
for some $C>0$, where we used \eqref{k_i est} in the second inequality. Since $\Phi=r^{\frac{\a}{\b}}H^\frac{1}{\b}$, we obtain
\begin{align*}
|\nabla\Phi|^2=|r^{\frac{\a}{\b}}\nabla H^{\frac{1}{\b}}+\frac{\a}{\b}r^{\frac{\a}{\b}-1}H^{\frac{1}{\b}}\nabla r|^2=r^{\frac{2\a}{\b}}|\nabla H^{\frac{1}{\b}}|^2+2\frac{\a}{\b}r^{\frac{2\a}{\b}-1}H^{\frac{1}{\b}}\nabla_i H^{\frac{1}{\b}}\nabla_i r+\frac{\a^2}{\b^2}r^{\frac{2\a}{\b}-2}H^{\frac{2}{\b}}|\nabla r|^2.
\end{align*}
Therefore, for the third line of \eqref{LQ 2}, we get
\begin{align}\label{LQ 3l}
&-\frac{a}{c}\f 1-\frac{a}{c}\r\frac{1}{\b}r^{\frac{\a}{\b}}H^{\frac{1}{\b}-1}\frac{|\nabla\Phi|^2}{\f\Phi-a\r^2}-2\gamma +\f 2-\frac{a}{c} \r \frac{\a}{\b}\frac{\Phi^2u}{r^2(\Phi-a)}
-\f 2-\frac{a}{c} \r \frac{(\a-1)\gamma}{\b}\frac{\Phi}{\Phi-a}\nonumber\\
=&-\frac{a}{c}\f 1-\frac{a}{c}\r\frac{1}{\b}r^{\frac{3\a}{\b}}H^{\frac{1}{\b}-1}\frac{|\nabla H^{\frac{1}{\b}}| ^2}{\f\Phi-a\r^2}- \frac{a}{c}\f 1-\frac{a}{c}\r \frac{2\a}{\b^2}r^{\frac{3\a}{\b}-1}H^{\frac{2}{\b}-1}\frac{ \metric{\nabla H^{\frac{1}{\b}}}{\nabla r }}{\f\Phi-a\r^2}\nonumber\\
&- \frac{a}{c}\f 1-\frac{a}{c}\r \frac{\a^2}{\b^3}r^{\frac{3\a}{\b}-2}H^{\frac{3}{\b}-1}\frac{|\nabla r|^2}{\f\Phi-a\r^2}-2\gamma +\f 2-\frac{a}{c} \r \frac{\a}{\b}\frac{\Phi^2u}{r^2(\Phi-a)}
-\f 2-\frac{a}{c} \r \frac{(\a-1)\gamma}{\b}\frac{\Phi}{\Phi-a}\nonumber\\
\leq& -c_2 |\nabla H^{\frac{1}{\b}}|^2 +C |\nabla H^{\frac{1}{\b}}| +C
\end{align}
for some constants $C$, $c_2>0$.
Combining \eqref{LQ 1l}, \eqref{LQ 2l} with \eqref{LQ 3l}, we arrive at
\begin{align*}
\mathcal{L} Q \leq -c_1|A|^2-c_2|\nabla H^{\frac{1}{\b}}| ^2+C|A|+C+C\frac{1}{|A|}+C|\nabla H^{\frac{1}{\b}}| +C\frac{|\nabla H^{\frac{1}{\b}}| ^2}{|A|}
\end{align*}
where $C$, $c_1$, $c_2$ only depend on the $C^0$, $C^1$ estimates and the bounds of $H$.
If $|A| > \frac{2}{c_2}$ at p, we have 
\begin{align*}
\mathcal{L}Q \leq -c_1|A|^2+C|A|+C,
\end{align*}
which induces $Q$ is bounded along the flow, since $2B \log (\Phi-a)$ is bounded. Then, we get the uniform bound of the principal curvatures from the definition of $Q$.
\end{proof}

Before analyzing the case of $k\geq 2$, we first state a crucial lemma (see \cite[Lemma 3.2]{GLL11}). For readers' convenience, we give a proof here.
\begin{lem}\label{liyanyan}
If $\{h_i{}^j\}\in\Gamma^+_{k}$ and $k\geq2$, we have the following inequality,
\begin{equation}
\ddot{\s_k}^{pq,rs}\nabla_ih_{pq}\nabla_ih_{rs}\leq-\s_k \f \frac{\nabla_i\s_k}{\s_k}-\frac{\nabla_iH}{H} \r\f  \f \frac{2-k}{k-1} \r\frac{\nabla_i\s_k}{\s_k}- \f \frac{k}{k-1} \r\frac{\nabla_iH}{H} \r .
\end{equation}
\end{lem}
\begin{proof}
We know from \cite[p. 23]{And07} that $(\frac{\s_k}{H})^{\frac{1}{k-1}}$ is concave. So we have
\begin{align*}
0\geq&\frac{\partial^2}{\partial h_{pq}\partial h_{rs}}\f\frac{\s_k}{H}\r^{\frac{1}{k-1}}\nabla_ih_{pq}\nabla_ih_{rs}\\
=& \frac{\partial }{\partial h_{rs}}\f\frac{1}{k-1}\f\frac{\s_k}{H}\r^{\frac{1}{k-1}-1}\dot{\f\frac{\s_k}{H}\r}^{pq} \r \nabla_ih_{pq}\nabla_ih_{rs}\\
    =&\frac{1}{k-1}\f\frac{\s_k}{H}\r^{\frac{1}{k-1}-1} \frac{\partial }{\partial h_{rs}}\f\frac{\dot{\s_k}^{pq}}{H}-\frac{\s_k\dot{H}^{pq}}{H^2}\r \nabla_ih_{pq}\nabla_ih_{rs}\\
    &+\frac{1}{k-1}\f\frac{1}{k-1}-1\r\f\frac{\s_k}{H}\r^{\frac{1}{k-1}-2}\f\frac{\dot{\s_k}^{rs}}{H}-\frac{\s_k\dot{H}^{rs}}{H^2}\r\f\frac{\dot{\s_k}^{pq}}{H}-\frac{\s_k\dot{H}^{pq}}{H^2}\r \nabla_ih_{pq}\nabla_ih_{rs}.
\end{align*}
That is,
\begin{align*}
\frac{\ddot{\s_k}^{pq,rs}\nabla_ih_{pq}\nabla_ih_{rs}}{\s_k}-\frac{2\nabla_i\s_k\nabla_iH}{H\s_k}+\frac{2|\nabla H|^2}{H^2}+\f\frac{1}{k-1}-1\r\f\frac{\nabla_i\s_k}{\s_k}-\frac{\nabla_iH}{H}\r\f\frac{\nabla_i\s_k}{\s_k}-\frac{\nabla_iH}{H}\r\leq 0.
\end{align*}
Hence,
\begin{align*}
\ddot{\s_k}^{pq,rs}\nabla_ih_{pq}\nabla_ih_{rs}\leq-\s_k\f\frac{\nabla_i\s_k}{\s_k}-\frac{\nabla_iH}{H}\r\f \f\frac{2-k}{k-1}\r\frac{\nabla_i\s_k}{\s_k}-\f\frac{k}{k-1}\r\frac{\nabla_iH}{H}\r.
\end{align*}
\end{proof}

\begin{lem}\label{Lemma 4.5}
Under the flow \eqref{s1:flow-n}, when $k\geq2$ and $\b \in (0,1] \cup \{ k\} $, the principal curvatures of the $k$-convex solution have a uniform bound, i.e.
\begin{equation*}
|\kappa_i|\leq C \quad i=1,\cdots,n.
\end{equation*}
\end{lem}
\begin{proof}
As $\{h_i{}^j\}\in\Gamma^+_k$, $k \geq 2$, we have $H>0$, $\s_2>0$, then
\begin{equation*}
H^2=2\s_2+|A|^2>|A|^2.
\end{equation*}
So we only need to prove that $H$ has an upper bound.

We introduce a new auxiliary function
\begin{equation}\label{aux-fun-2}
 Q=\frac{H}{\Phi-a},
\end{equation}
where $\Phi=r^{\frac{\alpha}{\beta}}\s_k^{\frac{1}{\b}}$ and $a=\frac{1}{2}\inf_{M \times [0,T)}\Phi$. Using the evolution equations \eqref{evolution h} and \eqref{phi}, we obtain
\begin{equation}\label{LQ gen 1}
\begin{aligned}
\mathcal{L}Q=&\frac{\mathcal{L}H}{\Phi-a}  - \frac{H\mathcal{L}\Phi}{(\Phi-a)^2}+\frac{2r^{\frac{\a}{\b}} \dot{F}^{ij} \nabla_i H \nabla_j \Phi }{ \f \Phi-a \r^2} - \frac{2r^{\frac{\a}{\b}} H \dot{F}^{ij} \nabla_i \Phi \nabla_j \Phi }{ \f \Phi-a \r^3}\\
  =& - \f \frac{k}{\b}-1 \r \frac{\Phi}{\Phi-a} |A|^2 - \frac{a}{\f \Phi-a \r^2} r^{\frac{\a}{\b}} H \dot{F}^{pq} h_p{}^l h_{lq}+\frac{r^{\frac{\a}{\b}} \ddot{F}^{pq,rs} \nabla_i h_{pq} \nabla_i h_{rs}}{\Phi-a} \\
 &+ \frac{2\frac{\a}{\b} r^{\frac{\a}{\b}-1}}{\Phi-a} \nabla_i F \nabla_i r + \frac{\frac{\a}{\b} r^{\frac{\a}{\b}-1}F}{\Phi-a} \Delta r
+ \frac{\a}{\b} \f \frac{\a}{\b} - 1 \r r^{\frac{\a}{\b}-2} \frac{F |\nabla r|^2}{\Phi-a}\\
 &- \frac{\gamma}{\Phi-a} H
+ \frac{\a}{\b}\frac{\Phi^2u}{(\Phi-a)^2r^2}H-\frac{\f \a-k \r\gamma}{\b} H\frac{\Phi}{(\Phi-a)^2}\\
&+\frac{2r^{\frac{\a}{\b}} \dot{F}^{ij} \nabla_i H \nabla_j \Phi}{ \f \Phi-a \r^2}
-\frac{2r^{\frac{\a}{\b}} H \dot{F}^{ij} \nabla_i \Phi \nabla_j \Phi}{ \f \Phi-a \r^3}.
\end{aligned}
\end{equation}
Since $\s_k=F^{\b}$, we have
\begin{align*}
\dot{F}^{pq}=\frac{\dot{\s_k}^{pq}}{\b F^{\b-1}}, \quad
\ddot{F}^{pq,rs}\nabla_ih_{pq}\nabla_ih_{rs}=\frac{\ddot{\s_k}^{pq,rs}\nabla_ih_{pq}\nabla_ih_{rs}}{\b F^{\b-1}}-\frac{\b-1}{F}|\nabla F|^2.
\end{align*}
Therefore, using Lemma \ref{liyanyan}, we obtain

\begin{align}\label{LYY}
\ddot{F}^{pq,rs}\nabla_ih_{pq}\nabla_ih_{rs}\leq& -  \frac{F}{\b} \f \frac{\b \nabla_i F}{F}-\frac{\nabla_i H}{H}\r \f \frac{2-k}{k-1}\frac{\b \nabla_i F}{F}-\frac{k}{k-1}\frac{\nabla_i H}{H} \r -\frac{\b-1}{F}|\nabla F|^2\nonumber\\
                               =&\f 1-\frac{\b}{k-1} \r \frac{|\nabla F|^2}{F}+\frac{2}{k-1}\nabla_i F\frac{\nabla_i H}{H}-\frac{kF}{\b \f k-1\r}\frac{|\nabla H|^2}{H^2}.
\end{align}
 At the maximum point p of $Q$ on $M_t$, we have
\begin{align}\label{cr Q 2}
\frac{\nabla H}{H}= \frac{\nabla \Phi}{\Phi-a}=\frac{r^{\frac{\a}{\b}} \nabla F }{ \Phi-a }+\frac{\frac{\a}{\b}r^{\frac{\a}{\b}-1} F\nabla r}{ \Phi-a }.
\end{align}
Then the last two terms in \eqref{LQ gen 1} becomes
\begin{align}\label{LQ 2 cr}
\frac{2r^{\frac{\a}{\b}} \dot{F}^{ij} \nabla_i H \nabla_j \Phi}{ \f \Phi-a \r^2} -\frac{2r^{\frac{\a}{\b}} H \dot{F}^{ij} \nabla_i \Phi \nabla_j \Phi}{ \f \Phi-a \r^3}=0.
\end{align}
Inserting \eqref{cr Q 2} into \eqref{LYY}, the third term in \eqref{LQ gen 1} becomes
\begin{align}\label{LYY final}
&\frac{r^{\frac{\a}{\b}}}{\Phi-a}\ddot{F}^{pq,rs}\nabla_ih_{pq}\nabla_ih_{rs}\nonumber\\
\leq&\f 1-\frac{\b}{k-1} \r\frac{ r^{\frac{\a}{\b}} }{\Phi-a}\frac{|\nabla F|^2}{F}
+\frac{2}{k-1}\frac{r^{\frac{2\a}{\b}}}{\f \Phi-a\r ^2}|\nabla F|^2
+\frac{2\a}{\f k-1 \r\b}\frac{r^{\frac{2\a}{\b}-1}F }{\f \Phi-a\r ^2}\metric{\nabla F}{\nabla r}\nonumber\\
&-\frac{k}{\b \f k-1\r}\frac{r^{\frac{2\a}{\b}}\Phi}{\f \Phi-a \r ^3}|\nabla F|^2-\frac{2k\a}{\b^2 \f k-1\r}\frac{\Phi^2r^{\frac{\a}{\b}-1} }{\f \Phi-a \r ^3}\metric{\nabla F}{\nabla r} -\frac{k\a^2}{\b^3 \f k-1\r}\frac{\Phi^2r^{\frac{\a}{\b}-2}F}{\f \Phi-a \r ^3}|\nabla r|^2 \nonumber\\
\leq&\f \f 1-\frac{\b}{k-1} \r +\frac{2}{k-1}\frac{\Phi}{\Phi-a}-\frac{k}{\b \f k-1\r}\frac{\Phi^2}{\f\Phi-a\r ^2} \r \frac{r^{\frac{\a}{\b}}|\nabla F|^2}{F\f \Phi-a \r } \nonumber\\
&+\frac{2\a}{\f k-1 \r\b}\frac{r^{\frac{2\a}{\b}-1}F }{\f \Phi-a\r ^2}|\nabla F||\nabla r|
 +\frac{2k\a}{\b^2 \f k-1\r}\frac{\Phi^2r^{\frac{\a}{\b}-1} }{\f \Phi-a \r ^3}|\nabla F||\nabla r| -\frac{k\a^2}{\b^3 \f k-1\r}\frac{\Phi^2r^{\frac{\a}{\b}-2}F}{\f \Phi-a \r ^3}|\nabla r|^2\nonumber\\
\leq&\f - \frac{k}{\b(k-1)} \f\frac{a}{\Phi-a} \r^2 + \frac{2(\b-k)}{\b(k-1)} \frac{a}{\Phi-a}- \frac{(1-\b)(k-\b)}{\b(k-1)} \r \frac{r^{\frac{\a}{\b}}|\nabla F|^2}{F\f \Phi-a \r } +C|\nabla F| +C
\end{align}
for some $C>0$ as $\Phi$ and $r$ are bounded along the flow \eqref{s1:flow-n}. From the definitions of $a$ and $c$, $\frac{a}{c-a} \leq \frac{a}{\Phi-a} \leq 1$. Thus, from the assumption $\b \in (0,1] \cup \{k\}$, we have
\begin{align*}
\frac{r^{\frac{\a}{\b}}}{\Phi-a}\ddot{F}^{pq,rs}\nabla_ih_{pq}\nabla_ih_{rs}
\leq -c' |\nabla F|^2 +C|\nabla F| +C
\end{align*}
for some positive constants $c'$ and $C$.
At the same time, we have
\begin{align}\label{-|A|^2<0}
-\f\frac{k}{\b}-1\r\frac{\Phi}{\Phi-a}|A|^2 \leq 0.
\end{align}
From \eqref{rij}, we obtain
\begin{align*}
\Delta r=\frac{n}{r}-\frac{u}{r}H-\frac{|\nabla r|^2}{r}.
\end{align*}
Then
\begin{align*}
&\frac{2\frac{\a}{\b} r^{\frac{\a}{\b}-1}}{\Phi-a} \nabla_i F \nabla_i r + \frac{\frac{\a}{\b} r^{\frac{\a}{\b}-1}F}{\Phi-a} \Delta r
+ \frac{\a}{\b} \f \frac{\a}{\b} - 1 \r r^{\frac{\a}{\b}-2} \frac{F |\nabla r|^2}{\Phi-a} \\
\leq&\frac{2\frac{\a}{\b}r^{\frac{\a}{\b}-1}}{\Phi-a}|\nabla F||\nabla r|+\frac{\frac{n\a}{\b}r^{\frac{\a}{\b}-2}F}{\Phi-a}-\frac{\frac{\a}{\b}r^{\frac{\a}{\b}-2}HuF}{\f \Phi-a\r}+\frac{\a}{\b}\f\frac{\a}{\b}-2\r r^{\frac{\a}{\b}-2}\frac{F|\nabla r|^2}{\Phi-a}\\
\leq& C|\nabla F|+CH+C
\end{align*}
for some constant $C>0$. At $p$, we diagonalize the second fundamental form and denote $f(\kappa)=F(h_i^j)$. By using \eqref{LQ 2 cr}, \eqref{LYY final} and \eqref{-|A|^2<0}, the evolution equation \eqref{LQ gen 1} of $Q$ becomes,
\begin{equation}\label{LQ-F}
\begin{aligned}
\mathcal{L}Q \leq&-\frac{a}{\f\Phi-a\r^2}r^{\frac{\a}{\b}}H \sum_i \dot{f}^i\kappa_{i}{}^2-c'|\nabla F|^2+C|\nabla F|+CH+C\\
\leq& -\frac{a}{\f\Phi-a\r^2}r^{\frac{\a}{\b}}H \sum_i \dot{f}^i\kappa_{i}{}^2+CH+C
\end{aligned}
\end{equation}
for some $C>0$.
Without loss of generality, we assume that $\kappa_1 \geq \kappa_2, \cdots \geq \kappa_n$. Since $\lbrace h_{i}^j \rbrace \in \Gamma^+_k$, by employing the properties of the symmetric functions listed in subsection \ref{symm}, we have
$\dot{f}^i=\frac{\dot{\s}_k{}^i}{\b f^{\b-1}}>0$ and
\begin{align*}
\dot{f}^1\kappa_1=\frac{\dot{\s}_k{}^1\kappa_1}{\b f^{\b -1}}\geq\frac{\frac{k}{n}\s_k}{\b f^{\b -1}}.
\end{align*}
Thus,
\begin{align*}
-\frac{a}{\f\Phi-a\r^2}r^{\frac{\a}{\b}}H \sum_i \dot{f}^i\kappa_{i}{}^2
\leq -\frac{a}{\f\Phi-a\r^2}r^{\frac{\a}{\b}} H(\dot{f}^1 \k_1) \k_1
\leq -c_1 H^2
\end{align*}
for some $c_1>0$, where we used $H \leq n \k_1$. Then from \eqref{LQ-F}, we have
\begin{align*}
\mathcal{L}Q \leq& -c_1H^2+CH+C
\end{align*}
for some $C>0$,
which implies an upper bound of $Q$. Hence, $H$ is bounded along the flow \eqref{s1:flow-n} and we complete the proof of Lemma \ref{Lemma 4.5}.
\end{proof}
At last, we will prove a $C^2$ estimate of \eqref{s1:flow-n} for the convex case. In this part, our method is motivated by \cite{LSW20b}. To estimate the lower bound of the principal curvatures, we need the following lemma.
\begin{lem}\label{lem-inv-con}
	Let $\{\tilde{h}^{ij} 
	\}$ be the inverse of $\{h_{ij}
	\}$. Then $\lbrace\tilde{h}_i{}^j\rbrace$ represents the inverse of Weingarten map. If $\lbrace h_i{}^j \rbrace> 0$, we have
	\begin{align}
	(\ddot{G}^{pq,lm} + 2 \dot{G}^{pm}\tilde{h}_q{}^l)\eta_p{}^q \eta_l{}^m \geq 2 G^{-1} (\dot{G}^{pq} \eta_p{}^q)^2\label{f-inv-con}
	\end{align}
	for any tensor $\{\eta_p{}^q\}$, where $G=\s_k^{\frac{1}{k}}(h_i^j)$.
\end{lem}
\begin{proof}
	In the proof of this lemma, we denote $\dot{G}^{pq} := \frac{\partial G}{\partial h_p{}^q}$ and $\ddot{G}^{pq,rs} := \frac{\partial^2 G}{\partial h_p{}^q \partial h_r{}^s}$. Assume $S(\tilde{h}_i{}^j) = \frac{1}{G(h_i{}^j)}$, then
	\begin{align*}
	\dot{S}^{ij} := \frac{\partial S}{\partial \tilde{h}_i{}^j} = - G^{-2} \dot{G}^{pq} \frac{\partial h_p{}^q}{ \partial \tilde{h}_i{}^j} =
	G^{-2}\dot{G}^{pq}h_p{}^i h_j{}^q.
	\end{align*}
	Since $G$ is inverse-concave, i.e., $S$ is concave, given any $\bar{\eta}_i{}^j$, we have
	\begin{align}
	0 \geq& S^{ij, \lambda \mu} \bar{\eta}_i{}^j \bar{\eta}_\lambda{}^\mu \nonumber\\
	=& \frac{\partial }{\partial h_l{}^m}(G^{-2}\dot{G}^{pq}h_p{}^i h_j{}^q) \frac{\partial h_l{}^m}{\partial \tilde{h}_\lambda{}^\mu} \bar{\eta}_i{}^j \bar{\eta}_\lambda{}^\mu \nonumber\\
=& -h_l{}^\lambda h_\mu{}^m(-2 G^{-3}\dot{G}^{lm} \dot{G}^{pq}h_p{}^i h_j{}^q + G^{-2}\ddot{G}^{pq,lm} h_p{}^i h_j{}^q + G^{-2}\dot{G}^{pq} \d_p{}^l \d_m{}^i h_j{}^q \nonumber\\
&+ G^{-2}\dot{G}^{pq}h_p{}^i \d_j{}^l \d_m{}^q)\bar{\eta}_i{}^j \bar{\eta}_\lambda{}^\mu \nonumber\\	
=& 2G^{-3}\dot{G}^{lm}\dot{G}^{pq}h_p{}^i h_j{}^q h_l{}^\lambda h_\mu{}^m\bar{\eta}_i{}^j \bar{\eta}_\lambda{}^\mu - G^{-2}\ddot{G}^{pq,lm} h_p{}^i h_j{}^q h_l{}^\lambda h_\mu{}^m\bar{\eta}_i{}^j \bar{\eta}_\lambda{}^\mu\nonumber\\
&-2G^{-2} \dot{G}^{pq} h_p{}^i h_j{}^m \tilde{h}_{m}{}^l h_l{}^\lambda h_\mu{}^q \bar{\eta}_i{}^j \bar{\eta}_{\lambda}{}^\mu \nonumber,
	\end{align}
where we used $\d_m{}^\lambda = \tilde{h}_m{}^l h_l{}^\lambda$ in the last equality of the above equation. Now, by setting $\eta_p{}^q = h_p{}^i \bar{\eta}_i{}^j h_j{}^q$, we have
	\begin{align*}
	0 \geq 2 G^{-3} \dot{G}^{pq}\dot{G}^{lm} \eta_p{}^q \eta_l{}^m - G^{-2} \ddot{G}^{pq,lm}\eta_p{}^q \eta_l{}^m - 2G^{-2}\dot{G}^{pq} \tilde{h}_{m}{}^l \eta_p{}^m \eta_l{}^q,
	\end{align*}
	which implies the statement.
	\end{proof}
\begin{lem}\label{conv-c2}
	Let $X(\cdot, t)$ be a smooth, closed and uniformly convex solution to the normalized flow \eqref{s1:flow-n} for $t \in [0,T)$, which enclosed the origin. If $\a \geq \b+k$, $\b > 0$, there exists a positive constant $C$ depending only on $\a$ and $M_0$, 
	such that the principal curvatures of $X (\cdot, t)$ satisfy
	\begin{align*}
	\frac{1}{C} \leq \k_i(\cdot, t) \leq C , \quad \forall \ t \ \in [0,T) \ {\rm and} \ i=1,2,\cdots,n.
	\end{align*}
\end{lem}
\begin{proof}
	Let  $\lambda(x,t)$ denote the maximal principal radii at $X(x,t)$.
 Fix an arbitrary $t_0 \in [0,T)$, we choose a point $x_0 \in M_{t_0}$ such that $\lambda(x_0,t_0) = \max_{M_{t_0}} \lambda(\cdot, t_0)$. Then we can choose a normal coordinate system near $(x_0,t_0) $ on $M_{t_0}$ diagonalizing the second fundamental form at $(x_0,t_0) $. Further, we can assume $\partial_1 |_{(x_0,t_0)}$ is an eigenvector with respect to $\lambda(x_0,t_0)$, i.e., $\lambda(x_0,t_0) = \tilde{h}_1{}^1 (x_0, t_0)$. Using this coordinate system, we can calculate the evolution of $\tilde{h}_1{}^1$ at $(x_0,t_0)$. From
	\begin{equation}\label{par t tilde h}
	\frac{\partial}{\partial t}\tilde{h}_1{}^1=-( \tilde{h}_1{}^1 ) ^2\partial_t h_1{}^1,
	\end{equation}
	\begin{align*}
	\nabla_i \tilde{h}_1{}^1 =& \frac{\partial \tilde{h}_1{}^1}{\partial h_p{}^q} \nabla_i h_p{}^q
	= - \tilde{h}_1{}^p \tilde{h}_q{}^1 \nabla_i h_p{}^q
	= - (\tilde{h}_1{}^1)^2 \nabla_i h_{11},
	\end{align*}
and
\begin{equation}\label{2 deriv tilde h}
	\begin{aligned}
	\nabla_j \nabla_i \tilde{h}_1{}^1 =& \nabla_j (- \tilde{h}_1{}^p \tilde{h}_q{}^1 \nabla_i h_p{}^q) \\
	=& -\nabla_j \tilde{h}_1{}^p \tilde{h}_q{}^1 \nabla_i h_p{}^q - \tilde{h}_1{}^p  \nabla_j \tilde{h}_q{}^1 \nabla_i h_p{}^q - \tilde{h}_1{}^p \tilde{h}_q{}^1 \nabla_j \nabla_i h_p{}^q\\
	 =& \tilde{h}_1{}^r\tilde{h}_l{}^p \tilde{h}_q{}^1 \nabla_j h_r{}^l  \nabla_i h_p{}^q + \tilde{h}_1{}^p \tilde{h}_r{}^1 \tilde{h}_q{}^s \nabla_j h_s{}^r \nabla_i h_p{}^q - \tilde{h}_1{}^p \tilde{h}_q{}^1 \nabla_j \nabla_i h_p{}^q \\
	=& - (\tilde{h}_1{}^1)^2 \nabla_j \nabla_i h_1{}^1 + 2 (\tilde{h}_1{}^1)^2 \tilde{h}^{pl} \nabla_i h_{1p} \nabla_j h_{1l},
	\end{aligned}
\end{equation}
we have by using \eqref{par t tilde h} and \eqref{2 deriv tilde h}
\begin{equation}\label{L tilde h 11}
	\begin{aligned}
	\mathcal{L}\tilde{h}_1{}^1=&\partial_t\tilde{h}_1{}^1-r^{\frac{\a}{\b}}\dot{F}^{ij}\nabla_j\nabla_i\tilde{h}_1{}^1\\
	=& -( \tilde{h}_1{}^1 )^2\partial_t h_1{}^1+ (\tilde{h}_1{}^1)^2r^{\frac{\a}{\b}}\dot{F}^{ij} \nabla_j \nabla_i h_1{}^1 -2 (\tilde{h}_1{}^1)^2 r^{\frac{\a}{\b}}\dot{F}^{ij} \tilde{h}^{pq} \nabla_i h_{1p} \nabla_j h_{1q}\\
	=&-( \tilde{h}_1{}^1 )^2 \mathcal{L} h_1{}^1-2 (\tilde{h}_1{}^1)^2 r^{\frac{\a}{\b}}\dot{F}^{ij} \tilde{h}^{pq} \nabla_i h_{1p} \nabla_j h_{1q}.
	\end{aligned}
\end{equation}
Plugging \eqref{evolution h} into \eqref{L tilde h 11}, we obtain
	\begin{align*}
	\mathcal{L}\tilde{h}_1{}^1=&\f\frac{k}{\b}-1\r\Phi-r^{\frac{\a}{\b}}\dot{F}^{pq}h_{lp}h_{lq}\tilde{h}_1{}^1-r^{\frac{\a}{\b}}(\tilde{h}_1{}^1)^2\ddot{F}^{pq,rs}\nabla_1h_{pq}\nabla_1h_{rs}\\
&-2\frac{\a}{\b}r^{\frac{\a}{\b}-1}(\tilde{h}_1{}^1)^2\nabla_1F\nabla_1r-\frac{\a}{\b}r^{\frac{\a}{\b}-1}F(\tilde{h}_1{}^1)^2\nabla_1\nabla_1r-\frac{\a}{\b}\f\frac{\a}{\b}-1\r r^{\frac{\a}{\b}-2}F(\tilde{h}_1{}^1)^2\nabla_1r\nabla_1r+\gamma\tilde{h}_1{}^1\\
&-2 (\tilde{h}_1{}^1)^2 r^{\frac{\a}{\b}}\dot{F}^{ij} \tilde{h}^{pq} \nabla_i h_{1p} \nabla_j h_{1q}.
	\end{align*}
Then we denote $G=\s_k^{\frac{1}{k}}$, thus $F=G^{\frac{k}{\b}}$. We have
$$
\dot{F}^{pq}=\frac{k}{\b}G^{\frac{k}{\b}-1}\dot{G}^{pq},
$$
$$
\ddot{F}^{pq,rs}=\frac{k}{\b}G^{\frac{k}{\b}-1}\ddot{G}^{pq,rs}+\frac{k}{\b}\f \frac{k}{\b}-1 \r G^{\frac{k}{\b}-2}\dot{G}^{pq}\dot{G}^{rs}.
$$
Direct computation gives
\begin{equation}\label{tildeh}	
	\begin{aligned}
	\mathcal{L}\tilde{h}_1{}^1=&\f\frac{k}{\b}-1\r\Phi-\sum_i \frac{k}{\b}\Phi G^{-1}\dot{G}^{ii}h_{ii}^2\tilde{h}_1{}^1-\frac{k}{\b}\Phi G^{-1}(\tilde{h}_1{}^1)^2\ddot{G}^{pq,rs}\nabla_1h_{pq}\nabla_1h_{rs}\\
	&-\frac{k}{\b}\f \frac{k}{\b}-1 \r \Phi G^{-2}(\tilde{h}_1{}^1)^2 \f \nabla_1 G \r^2
-2\frac{\a k}{\b^2}\Phi r^{-1}G^{-1}(\tilde{h}_1{}^1)^2\nabla_1 G\nabla_1r\\
&-\frac{\a}{\b}r^{-1}\Phi(\tilde{h}_1{}^1)^2\nabla_1\nabla_1r-\frac{\a}{\b}\f\frac{\a}{\b}-1\r r^{-2}\Phi(\tilde{h}_1{}^1)^2\nabla_1r\nabla_1r+\gamma\tilde{h}_1{}^1\\
&-2 (\tilde{h}_1{}^1)^2 \frac{k}{\b}\Phi G^{-1}\dot{G}^{ij} \tilde{h}^{pq} \nabla_i h_{1p} \nabla_j h_{1q}.
	\end{aligned}
\end{equation}	
Letting $\eta_i{}^j =\nabla_1 h_i{}^j$ in Lemma \ref{lem-inv-con}, by the Codazzi equation, we have
\begin{equation}\label{use inv-con}
\begin{aligned}
&\frac{k}{\b}\Phi G^{-1}(\tilde{h}_1{}^1)^2\ddot{G}^{pq,rs}\nabla_1h_{pq}\nabla_1h_{rs}
+2 (\tilde{h}_1{}^1)^2 \frac{k}{\b}\Phi G^{-1}\dot{G}^{ij} \tilde{h}^{pq} \nabla_i h_{1p} \nabla_j h_{1q}\\
\geq&
2\frac{k}{\b} \Phi G^{-2}(\tilde{h}_1{}^1)^2 \f \nabla_1 G \r^2
\end{aligned}
\end{equation}	
Putting \eqref{use inv-con} into \eqref{tildeh}, we obtain
\begin{equation}\label{tildeh2}
	\begin{aligned}
	\mathcal{L}\tilde{h}_1{}^1\leq &\f\frac{k}{\b}-1\r\Phi-\sum_i \frac{k}{\b}\Phi G^{-1}\dot{G}^{ii}h_{ii}^2\tilde{h}_1^1+\gamma\tilde{h}_1^1-\frac{k}{\b}\f \frac{k}{\b}+1 \r \Phi G^{-2}(\tilde{h}_1{}^1)^2 \f \nabla_1 G \r^2\\
	&-2\frac{\a k}{\b^2}\Phi r^{-1}G^{-1}(\tilde{h}_1{}^1)^2\nabla_1 G\nabla_1r
-\frac{\a}{\b}r^{-1}\Phi(\tilde{h}_1{}^1)^2\nabla_1\nabla_1r-\frac{\a}{\b}\f\frac{\a}{\b}-1\r r^{-2}\Phi(\tilde{h}_1{}^1)^2\f \nabla_1r\r ^2.	
	\end{aligned}
\end{equation}
By use of
	\begin{align*}
	-2\frac{\a k}{\b^2}\Phi r^{-1}G^{-1}(\tilde{h}_1{}^1)^2\nabla_1 G\nabla_1r \leq \frac{k}{\b}\f \frac{k}{\b}+1 \r \Phi G^{-2}(\tilde{h}_1{}^1)^2 \f \nabla_1 G \r^2+\frac{\a^2 k}{\b^2\f \b+k\r} r^{-2}\Phi(\tilde{h}_1{}^1)^2\f \nabla_1r\r ^2,
	\end{align*}
we have from \eqref{tildeh2}
\begin{equation}\label{tildeh3}
	\begin{aligned}
	\mathcal{L}\tilde{h}_1{}^1\leq& \f\frac{k}{\b}-1\r\Phi+\gamma\tilde{h}_1{}^1-\frac{\a}{\b}r^{-1}\Phi(\tilde{h}_1{}^1)^2\nabla_1\nabla_1 r-\frac{\a\f \a-\b-k\r }{\b \f \b+k \r } r^{-2}\Phi(\tilde{h}_1{}^1)^2\f \nabla_1r\r ^2\\
	\leq& \f\frac{k}{\b}-1\r\Phi+\gamma\tilde{h}_1{}^1-\frac{\a}{\b}r^{-1}\Phi(\tilde{h}_1{}^1)^2\nabla_1\nabla_1 r,
	\end{aligned}
	\end{equation}
	where we used  $\dot{G}^{ii} >0$.
From \eqref{ri} and \eqref{rij}, we have
\begin{align}\label{hess11 r}
1- (\nabla_1 r)^2 \geq \frac{u^2}{r^2},
\quad
\nabla_1 \nabla_1 r=\frac{1}{r}-\frac{u}{r}h_{11}-\frac{\f \nabla_1 r \r ^2}{r}
\geq \frac{u^2}{r^3} - \frac{u}{r} h_{11}.
\end{align}
Inserting \eqref{hess11 r} into \eqref{tildeh3}, we have
\begin{align*}
	\mathcal{L}\tilde{h}_1{}^1\leq &\f\frac{k}{\b}-1\r\Phi+\f \gamma+\frac{\a}{\b}ur^{-2}\Phi \r \tilde{h}_1{}^1-\frac{\a}{\b} u^2 r^{-4}\Phi(\tilde{h}_1{}^1)^2.
\end{align*}
From Lemma \ref{C0-est}, Lemma \ref{c1} and Corollary \ref{cor Phi}, we derive that 
	\begin{align*}
	\frac{\partial }{\partial t}\tilde{h}_1{}^1\leq -c_1(\tilde{h}_1{}^1)^2+c_2\tilde{h}_1{}^1+c_3.
\end{align*}
for some constants $c_1$, $c_2$, $c_3$ at $(x_0,t_0)$.
Hence $\lambda(x,t)$ has a uniform upper bound, which means that the principal curvatures are bounded from below by a positive constant $c_4$. Meanwhile, by Lemma \ref{F2}, one can get an upper bound of $\s_k$.
From the convexity assumption, we have
\begin{align*}
C\geq \s_k=&\kappa_{\max}\s_{k-1}(\kappa|\kappa_{\max})+\s_k(\kappa|\kappa_{\max})\\
\geq& C_{n-1}^{k-1}\kappa_{\min}^{k-1}\kappa_{\max}\\
\geq& C_{n-1}^{k-1} c_4^{k-1}\kappa_{\max}.
\end{align*}
for some constant $C$.
 Hence, the principal curvatures are bounded from above. This completes the proof of Lemma \ref{conv-c2}.
\end{proof}
  \begin{rem}
 By a similar argument, we can also get the bound of principal curvatures by exchanging  $\s_k^{\frac{1}{\b}}(\k)$ to a general curvature function $f^\frac{k}{\b}(\k)$ in Lemma \ref{conv-c2}, where  $f(\k)$ is a smooth symmetric positive function defined on the positive cone $\Gamma^+$ and satisfies the following conditions:
\begin{enumerate}
	\item strictly increasing: $\dot{f}^i = \frac{\partial f}{\partial \k_i}>0$ for each $i= 1, 2, \cdots,n$,
	\item homogeneous of degree one: $f(\l \k) = \l f(\k)$ for any $\l >0$,
	\item  concave,
	\item inverse-concave,
	\item $f_{-1}$ vanishes on the boundary of $\Gamma^+$.
\end{enumerate}
As a consequence, we can prove a similar result to Theorem \ref{main-thm-4}. Because the argument is similar, we omit the explicit proof of the conclusion.
\end{rem}
Now we have obtained the priori estimates of flow \eqref{s1:flow-n}, with the initial hypersurfaces mentioned in Theorem \ref{main-thm-3}, Theorem \ref{main-thm-2}, Theorem \ref{main-thm-1} and Theorem \ref{main-thm-4}. From Lemma \ref{C0-est}, Lemma \ref{c1} and Lemma \ref{F}, these flows have  short time existence. Using the $C^2$ estimates given in Lemma \ref{Lem 4.3}, Lemma \ref{Lemma 4.5} and Lemma \ref{lem-inv-con}, due to Krylov \cite{Kryl87}, we can get the $C^{2,\lambda}$ estimate of the scalar equation \eqref{s1:flow-rn}. 
Hence, we get the long time existence of these flows.
\begin{lem}\label{longexsits}
	The smooth solution of \eqref{s1:flow-n} with the initial hypersuface mentioned in Theorem \ref{main-thm-3}, Theorem \ref{main-thm-2}, Theorem \ref{main-thm-1} or Theorem \ref{main-thm-4} exists for all time $t \in [0, +\infty)$.
\end{lem}
\section{Asymptotic convergence and proofs of Theorem \ref{main-thm-3}--\ref{main-thm-4}}\label{sec:5}
In this section, we improve the $C^1$ estimates in Lemma \ref{c1}.
\begin{lem}\label{exp-decay}
	Let $r(\cdot, t)$ be the solution to \eqref{s1:flow-n} on $\mathbb{S}^n \times [0,T)$. We admit the priori estimates given in Section \ref{sec:3}, \ref{sec:4} for various situations and assume that the corresponding hypersurfaces $M_t$ are $k$-convex along the normalized flow. If  $\a \geq \b+k$, $\b>0$ and $n \geq 2$, there exist positive constants $C$ and $a$ only depend on $M_0$, $\a$, $\b$, $k$, such that
	\begin{equation*}
	\frac{|Dr(\cdot,t)|}{r(\cdot,t)}\leq Ce^{-at}, \quad \forall \ t>0.
	\end{equation*}
\end{lem}
\begin{proof}
	Same as the proof of Lemma \ref{c1}, we only need to calculate at the maximum point p of $|D \g|^2$ for each time.
	
	Case 1: $\a>\b+k$. By \eqref{C1-final-ineq}, we obtain
	\begin{equation*}
	\frac{\partial}{\partial t}(\frac{1}{2}|D\g|^2)\leq-  \f  \frac{(\a-\b-k)}{\b}e^{(\frac{\a}{\b}-1)\g}\sqrt{1+|D\g|^2}\s_k{}^{\frac{1}{\b}}  \r|D\g|^2.
	\end{equation*}
	 Then from the priori estimates, we know $\partial_ t(\frac{1}{2}|D\g|^2)\leq -a |D \g|^2$ for some $a>0$ and get the exponential convergence of the flow \eqref{s1:flow-n}.
	
	Case 2: $\a=\b+k$. By \eqref{C1-final-ineq}, we have
	\begin{equation}\label{asy-1}
	\frac{\partial}{\partial t}(\frac{1}{2}|D\g|^2)\leq-\frac{1}{\b} e^{\frac{\a}{\b}\g}\s_k{}^{(\frac{1}{\b}-1)}\s_k{}^{pq}(\d_{pq}|D\g|^2-\g_p\g_q).
	\end{equation}
	 At p, we have from \eqref{rg-h}
	 \begin{equation}\label{asy}
	\begin{aligned}
	h_i{}^j
	=\frac{1}{r\rho}(\d_{ij}-\g_{ij}+\frac{\g_i{}^l\g_l\g_j}{\rho^2})
	=\frac{1}{r\rho}(\d_{ij}-\g_{ij}),
	\end{aligned}
	\end{equation}
	where we used $\g_i{}^l\g_l=0$ at p.
	Multiplying the both sides of \eqref{asy} by $\g_j$, we have
	\begin{align*}
	h_i^j\g_j=\frac{1}{r\rho}(\g_i-\g_{ij}\g_j)=\frac{1}{r\rho}\g_i,
	\end{align*}
	which means $e_1 = \frac{D \g}{|D \g|}$ is an eigenvector of $\lbrace h_i^j \rbrace$ with eigenvalue $\frac{1}{r \rho}$. For convenience, we choose a local orthonormal frame filed $\{e_1, e_2, \cdots, e_n\}$ such that $\lbrace h_i^j \rbrace$ is diagonal at p. Then from \eqref{rg-g}, we have
	\begin{equation}\label{asy-2}
		\begin{aligned}
	&e^{\frac{\a}{\b}\g}\s_k{}^{(\frac{1}{\b}-1)}\s_k{}^{pq}(\d_{pq}|D\g|^2-\g_p\g_q)\\
	=&e^{\frac{\a}{\b}\g}\s_k{}^{(\frac{1}{\b}-1)}\frac{\partial \s_k}{\partial h_p{}^s}g^{qs}(\d_{pq}|D\g|^2-\g_p\g_q)\\
	=&e^{\frac{\a}{\b}\g}\s_k{}^{(\frac{1}{\b}-1)}\frac{\partial \s_k}{\partial h_p{}^s}e^{-2\g}(\d_{qs}-\frac{\g_q\g_s}{\rho^2})(\d_{pq}|D\g|^2-\g_p\g_q)\\
	=&e^{(\frac{\a}{\b}-2)\g}\s_k{}^{(\frac{1}{\b}-1)}\frac{\partial \s_k}{\partial h_p{}^s}(\d_{ps}|D\g|^2-\g_p\g_s).
	\end{aligned}
	\end{equation}
	Now we use the principal curvatures $\k = (\k_1, \k_2, \cdots, \k_n)$ to estimate the right hand side of \eqref{asy-2},
	\begin{equation}\label{asy-3}
	\begin{aligned}
	&\frac{\partial \s_k}{\partial h_p{}^s}(\d_{ps}|D\g|^2-\g_p\g_s)\\
	=&
	|D \g|^2\sum_{i=1}^n\s_{k-1}(\kappa|i) - \sum_{p=1}^n\s_{k-1}(\kappa|p)\g_p{}^2\\
	=&  \f (n-k+1)\s_{k-1}(\kappa)-\s_{k-1}(\kappa)+\kappa_1\s_{k-2}(\kappa|1)  \r |D \g|^2\\
	=&
	\f  (n-k)\s_{k-1}(\kappa)+\kappa_1\s_{k-2}(\kappa|1)  \r |D \g|^2. 
	\end{aligned}
	\end{equation}
	From \eqref{asy-1}, \eqref{asy-2} and \eqref{asy-3}, we have
	\begin{align*}
	\frac{\partial}{\partial t}(\frac{1}{2}|D\g|^2)\leq-e^{(\frac{\a}{\b}-2)\g}\s_k{}^{(\frac{1}{\b}-1)}  \f  (n-k)\s_{k-1}(\kappa)+\kappa_1\s_{k-2}(\kappa|1) \r |D \g|^2.
	\end{align*}
Now we divide the discussion into two subcases.
	
Subcase 2.1: $k<n$. By Newton-MacLaurin inequality and the fact $\s_k \geq c$, we have $(\s_{k-1})^{\frac{1}{k-1}}\geq c_1 (\s_k)^{\frac{1}{k}}\geq c_2>0$ for some constants $c_1$ and $c_2$. Since $(\k|1) \in \Gamma^+_{k-1}$, we know $\s_{k-2} (\k|1) >0$. Then $\partial_t |D \g|^2 \leq -c |D \g|^2$ and the exponentially decay of $\frac{|D r|^2}{r^2} = |D \g|^2$ follows from standard argument.
	
Subcase 2.2: $k=n \geq 2$. In this situation, the principal curvatures have uniform upper and lower bounds by Lemma \ref{conv-c2}. Therefore $\partial_t |D \g|^2 \leq -c |D \g|^2$ for some $c$ at point $p$, where $c$ only depends on the priori estimates. Hence the lemma is true.
\end{proof}
\begin{rem}
	When $n=1$, $k =1$, $\a \geq 2$, $\max_{\mathbb{S}^n} |D r|$ exponentially converges to $0$. This fact was proved in \cite[Lemma 4.1]{LSW20b}.
\end{rem}
\begin{proof}[Proofs of Theorem \ref{main-thm-3}, Theorem \ref{main-thm-2}, Theorem \ref{main-thm-1} and Theorem \ref{main-thm-4}]
	 We remark that the main differences between these four theorems are $C^2$ estimates. The long time existence of the flow \eqref{s1:flow-n} has been shown in Lemma \ref{longexsits}.
   Now we prove that the hypersurfaces converge to a sphere along the normalized flow \eqref{s1:flow-n}. As $\max_{\mathbb{S}^n} |D r|$ converges to $0$ exponentially fast when $\a\geq \b+k$ and $n \geq 2$. we know $\max_{\mathbb{S}^n} |D^l  r|$ decays exponentially to $0$ for any $l \geq 1$ from the interpolation inequality and the priori estimates we have made.
	
	When $\a >\b+k$, we claim that the radial function $r$ converges to $1$ exponentially fast.
	 Since $M_0$ is star-shaped, it can be bounded by two spheres centred at the origin, we denote them as $\mathbb{S}^n(a_1)$ and $\mathbb{S}^n(a_2)$, $a_1 < a_2$. In the proof of Lemma \ref{C0-est}, we know the radial function of $M_t$ can be bounded by  $a_1(t)$ and $a_2(t)$ by the well-known comparison principle, where
	\begin{align*}
	a_i(t) = \f \frac{ e^{ \frac{\a-k-\b}{\b}  \gamma t}}{e^{ \frac{ \a-k-\b}{\b} \gamma t}- \frac{a_i^{\frac{\a-k-\b}{\b}}-1}{a_i^{\frac{\a-k-\b}{\b}}}  } \r^{\frac{\b}{\a-k-\b}}, \quad i =1,2.
	\end{align*}
	
	Actually, $\mathbb{S}^n(a_i(t))$ is the solution to the normalized flow \eqref{s1:flow-n} with initial hypersurface $\mathbb{S}^n(a_i)$, $i=1,2$. We see that $a_1(t)$ and $a_2(t)$ tend to $1$ as $t \to +\infty$, which implies that the radial function of $M_t$ converges to $1$ in $C^0$ norm. Combining the estimates of $|D r|$, we know $M_t$ converges to $\mathbb{S}^n(1)$ exponentially fast.
	
 When $\a =\b+k$ and $n \geq 2$,
	by Lemma \ref{exp-decay}, $r_{\max}(t) - r_{\min}(t)$ converge to $0$ exponentially fast. Using the monotonicity of $r_{\max}$ and $r_{\min}$ in $C^0$ estimate, we see that
	 $r$ converges to a constant as $t \to +\infty$. Again, exponential decay of $|D r|$ implies that the convergence hypersurface is a sphere.
\end{proof}

\section{A counter example}\label{sec:6}
In this section, we prove that if $\a< \b +k$, the flow \eqref{s1:flow-n} may have unbounded ratio of radii, that is
\begin{align*}
\mathcal{R}\f X\f \cdot, t \r \r:=\frac{\max_{\mathbb{S}^n} r(\cdot,t)}{\min_{\mathbb{S}^n} r(\cdot,t)}\rightarrow \infty \quad {\rm as} \quad  t\rightarrow T
\end{align*}
for some $T >0$. Since the construction of the counter example is almost the same as \cite[Section 5]{LSW20b}. In this section, we only give a sketch proof of Proposition \ref{ce}. For more details, we refer to \cite[Section 6]{LSW20a} 
 and \cite[Section 7]{Xiao20}.

First, we consider a function defined on  $B_1^n(0)\times[-1,0)$,
\begin{equation}
\phi(x,t)=\left\{\begin{aligned}
&-|t|^{\theta}+|t|^{-\theta+\s\theta}|x|^2, \quad  &{\rm if} \ |x|<|t|^{\theta},\\
&-|t|^{\theta}-\frac{1-\s}{1+\s}|t|^{\theta(1+\s)}+\frac{2}{1+\s}|x|^{1+\s}, \quad &{\rm if} \ |t|^{\theta}\leq |x|\leq 1,
\end{aligned}\right.
\end{equation}
where $\s=\frac{q\theta-\b}{k\theta}$, $q=-\a+\b+k>0$ and $\theta > \frac{1}{q}$ is a positive constant.

Next, for every $t \in (-1,0)$, we extend the graph of $\phi(\cdot,t)$ to a closed, convex, rotationally symmetric hypersurface $\hat{M}_t$ enclosing $B_1(z)$, where $z=(0,\cdots,0,10)$. By choosing $a$ large enough, we can see that $\hat{M}_t$ is a sub-solution to
\begin{equation}
\left\{\begin{aligned}
\frac{\partial}{\partial t}r(x,t)=& -ar^{\frac{\a}{\b}}\s_k(\kappa)^{\frac{1}{\b}},\\
r(\cdot,0)=& r_0(\cdot).
\end{aligned}\right.
\end{equation}

Then there exists a large constant $C$ such that $\hat{M}_t\subset B_C(0)$ for all $t\in(-1,0)$. Define $b=(2C)^{\frac{\a}{\b}}$, we construct a new flow,
$$
\frac{\partial}{\partial t}\tilde{X}(x,t)=-ba\tilde{r}^{\frac{\a}{\b}}\s_k(\kappa)^{\frac{1}{\b}} \nu(x,t),
$$
where $\tilde{r}=|X-z|$. We can choose $\tau\in(-1,0)$ close to 0 and the initial hypersurface of the flow as $\tilde{X}(\cdot,\tau)=\partial B_1^n(z)$, such that the flow enclosed $B_{\frac{1}{2}}(z)$ on time interval $(\tau,0)$. We select a hypersurface $M_0$ inside $\hat{M}_{\tau}$ and enclosing the ball $B_1(z)$. By the comparison principle, $M_{t_0}$ touch the origin at some $t_0\in(\tau,0)$ and meantime enclose $B_{\frac{1}{2}}(z)$.

Finally, rescaling $M_t$ to $\tilde{M}_t=a^{-\frac{1}{q}}M_t$, one easily verifies that $\tilde{M}_t$ is a solution to flow \eqref{flow-s} and satisfies the property \eqref{ace}. We complete the proof of Proposition \ref{ce}.

\end{document}